\documentclass{amsart}
\usepackage{latexsym,amsxtra,amscd,ifthen}
\usepackage{amsfonts}
\usepackage{verbatim}
\usepackage{amsmath}
\usepackage{amsthm}
\usepackage{amssymb}
\usepackage{color}
\usepackage{bbm}
\usepackage{xcolor}
\definecolor{navy}{HTML}{004d99}
\usepackage[matrix,arrow]{xy}
\usepackage{hyperref}
\hypersetup{
 colorlinks  = true,  
 urlcolor   = navy,  
 linkcolor  = navy,  
 citecolor  = black   
}

\numberwithin{equation}{section}

\theoremstyle{plain}
\newtheorem{theorem}{Theorem}[section]
\newtheorem{lemma}[theorem]{Lemma}
\newtheorem{proposition}[theorem]{Proposition}
\newtheorem{hypothesis}[theorem]{Hypothesis}

\newtheorem{corollary}[theorem]{Corollary}

\theoremstyle{definition}
\newtheorem{definition}[theorem]{Definition}
\newtheorem{example}[theorem]{Example}

\newtheorem{remark}[theorem]{Remark}
\newtheorem{question}[theorem]{Question}

\makeatletter       
\let\c@equation\c@theorem 
\makeatother

\newcommand{\otensor}{\mathbin{\overline{\otimes}}}
\newcommand{\ohom}{\overline{\Hom}}

\DeclareMathOperator{\Ext}{Ext}

\DeclareMathOperator{\depth}{depth}
\DeclareMathOperator{\projdim}{projdim}
\DeclareMathOperator{\injdim}{injdim}

\DeclareMathOperator{\Kdim}{Kdim}
\DeclareMathOperator{\gldim}{gldim}
\DeclareMathOperator{\GKdim}{GKdim}

\DeclareMathOperator{\Hom}{Hom}
\DeclareMathOperator{\RHom}{RHom}

\newcommand{\Modleft}{\textup{-Mod}}
\newcommand{\Modright}{\textup{Mod-}}
\newcommand{\Modleftfd}{\textup{-Mod}_{\text{fd}}}
\newcommand{\Modrightfd}{\textup{Mod}_{\text{fd}}\textup{-}}
\newcommand{\modleft}{\textup{-mod}}

\newcommand{\mc}{\mathcal}
\newcommand{\mb}{\mathbb}

\newcommand{\kk}{\Bbbk}

\begin{document}

\title[A proof of the Brown--Goodearl Conjecture for 
weak Hopf algebras]
{A proof of the Brown--Goodearl Conjecture for 
module-finite weak Hopf algebras}

\author[Rogalski]{Daniel Rogalski}
\address{(Rogalski) Department of Mathematics,
University of California, San Diego,
La Jolla, CA 92093, USA}
\email{drogalski@ucsd.edu}

\author[Won]{Robert Won}
\address{(Won) Department of Mathematics, Box 354350,
University of
Washington, Seattle, WA 98195, USA}
\email{robwon@uw.edu}

\author[Zhang]{James J. Zhang}
\address{(Zhang) Department of Mathematics, Box 354350,
University of
Washington, Seattle, WA 98195, USA}
\email{zhang@math.washington.edu}


\begin{abstract}
Let $H$ be a weak Hopf algebra that is a finitely generated 
module over its affine center. We show that $H$ has finite self-injective
dimension and so the Brown--Goodearl Conjecture holds in this special weak Hopf setting.
\end{abstract}


\maketitle

\setcounter{section}{-1}
\section{Introduction}
\label{xxsec0}
\noindent

In his Seattle lecture in 1997 \cite{Br1}, Brown posed several open questions about noetherian Hopf algebras which satisfy a polynomial identity. Since then, the homological 
properties of infinite-dimensional noetherian Hopf algebras have been 
investigated extensively. While quite a few of the questions in Brown's lecture 
have been answered, one important question, now called either 
the Brown--Goodearl Question or the Brown--Goodearl Conjecture, is still open:
\[
\textit{Does every noetherian Hopf algebra have finite injective dimension?}
\label{Q1}\tag{Q1}
\]
Or, asking for a slightly stronger property, 
\[
\textit{Is every noetherian Hopf algebra Artin--Schelter Gorenstein?}
\label{Q2}\tag{Q2}
\]
(The definition of Artin--Schelter Gorenstein will be recalled in Section~\ref{xxsec1}.) In recent years, the Brown--Goodearl Question has been posed in many lectures and 
survey papers \cite{Br1, Br2, Go}, as it is related to the existence 
of rigid dualizing complexes, and therefore related 
to the twisted Calabi--Yau property of these Hopf algebras. An affirmative 
answer to this question has many other consequences, especially in the 
study of the ring-theoretic properties of noetherian Hopf algebras. 

It is natural to ask the Brown--Goodearl Question for other classes of noetherian 
algebras that are similar to Hopf algebras, for example, weak 
Hopf algebras, braided Hopf algebras, Nichols algebras, and so on. In 
fact Andruskiewitsch independently asked the following question in 
2004 \cite[after Definition 2.1]{An}: if a Nichols algebra is a 
domain with finite Gelfand--Kirillov dimension, is it then 
Artin--Schelter regular (and therefore of finite global dimension)?

In this paper, we consider weak Hopf algebras, which are natural generalizations of Hopf algebras that 
have applications in conformal field theory, quantum field 
theory, and the study of operator algebras, subfactors, tensor categories, and fusion categories. One important fact is that any fusion category is 
equivalent to the category of modules over a weak Hopf algebra 
\cite{Ha1, Sz}, see also \cite[Theorem 2.20 and Corollary 2.22]{ENO}.
We refer the reader to Section~\ref{xxsec5} for the definition of a weak Hopf 
algebra and a few examples. Much of the existing literature on weak Hopf algebras has focused on the finite-dimensional case. For example, B{\" o}hm, Nill, and Szlach{\' a}nyi \cite[Theorem 3.11]{BNS} 
proved that every finite-dimensional 
weak Hopf algebra over a field is quasi-Frobenius, or equivalently, 
has (self-)injective dimension zero (see also \cite[Corollary 3.3]{Vec}). In contrast with the Hopf 
case, a finite-dimensional weak Hopf algebra may not be Frobenius
\cite[Proposition 2.5]{IK}. 

The main result of this paper is to affirmatively answer the 
analog of the Brown--Goodearl Question \eqref{Q1} for weak Hopf algebras that are module-finite 
over their affine centers. In fact, our proof works not just for weak Hopf algebras, 
but for any such algebra whose module category has a monoidal structure with certain basic properties.

In most of our results, we restrict our attention to algebras satisfying the following ring-theoretic hypothesis. 
\begin{hypothesis}
\label{xxhyp0.1}
Let $A$ be an algebra with center $Z(A)$. Suppose that 
$A$ is a finitely generated module over $Z(A)$ and that $Z(A)$ is 
a finitely generated algebra over the base field $\kk$. 
\end{hypothesis}

The favorable homological properties of Hopf algebras seem to arise primarily from the fact that there is an internal 
tensor product of modules over a Hopf algebra.  We abstract this idea here to study the following condition.
\begin{hypothesis}
\label{xxhyp0.2}
Let $A$ be an algebra over $\kk$.
Assume that there is a monoidal structure $\otensor$ on the category $A \Modleft$ of left $A$-modules, where $\otensor$ is bilinear on morphisms and biexact, such that every finite dimensional $M \in A \Modleft$ has a left dual $M^*$ in that monoidal category. 
Assume further that the same hypotheses hold for the category $A^{op}\Modleft$ of right $A$-modules.
\end{hypothesis}
We refer the reader to \cite{EGNO} for basic notions concerning monoidal categories; in any case, we will remind the reader of the undefined terms in Section~\ref{sec:tc}.
For any 
weak Hopf algebra $A$ which satisfies Hypothesis~\ref{xxhyp0.1}, the category $A \Modleft$ has a monoidal structure as in Hypothesis~\ref{xxhyp0.2}, where $M \otensor N$ is a subspace of the usual tensor product $M \otimes_\kk N$ (see \cite{BCJ,Nill}). 
If $A$ is a Hopf algebra or a quasi-bialgebra with antipode which satisfies Hypothesis~\ref{xxhyp0.1}, then $A$ also satisfies Hypothesis~\ref{xxhyp0.2} where $\otensor$ is the usual tensor product $\otimes_{\kk}$.

For a module $M$ over an algebra, let 
$\GKdim M$ denote the Gelfand--Kirillov dimension (or GK dimension) 
of $M$. (We refer to \cite{KL} for the definition of GK dimension.) We say an algebra $A$ is {\it homogeneous} if 
\[\GKdim L = \GKdim A
\]
for all nonzero left ideals $L\subseteq A$. 

Our main result is the following:
\begin{theorem}
\label{xxthm0.2}
Let $A$ be an algebra over a field $\kk$.
\begin{enumerate}
\item[(1)]
Assume Hypothesis~\ref{xxhyp0.1} and Hypothesis~\ref{xxhyp0.2} for $A$. Then $A$ has finite injective dimension. Further, as an algebra, $A$ is a finite direct sum of indecomposable noetherian algebras which are Artin--Schelter Gorenstein, Auslander Gorenstein, Cohen--Macaulay, and homogeneous of finite Gelfand--Kirillov dimension equal to their injective dimension. 
\item[(2)] If $H$ is a weak Hopf algebra, then $H$ satisfies Hypothesis~\ref{xxhyp0.2}.  In particular, if $H$ also satisfies Hypothesis~\ref{xxhyp0.1} then $H$ satisfies all of the conclusions of part (1).
\end{enumerate}
\end{theorem}
We remark that the direct sum of two Artin--Schelter Gorenstein weak Hopf algebras of different injective dimensions is a weak Hopf algebra which is not Artin--Schelter Gorenstein. Hence, this theorem gives the strongest answer to the analog of question \eqref{Q2} that is possible in this setting.

While we focus on weak Hopf algebras in this paper, we expect other generalizations of Hopf algebras to satisfy Theorem~\ref{xxthm0.2}(2).  As an example, we show this for quasi-bialgebras with antipode in Theorem~\ref{xxthm-qh} below.  It would be interesting to extend this theory to different kinds of Hopf-like structures.

We also note that if $A$ satisfies Hypothesis~\ref{xxhyp0.2} because it is a
weak Hopf algebra or other similar structure, as in Theorem~\ref{xxthm0.2}(2), the coproduct 
need not respect the direct sum decomposition in Theorem~\ref{xxthm0.2}(1).

Theorem~\ref{xxthm0.2} has many applications. Below we 
list some of the consequences of this theorem.  Undefined terminology 
will be reviewed in later sections.

\begin{theorem}
\label{xxthm0.3}
Assume Hypotheses~\ref{xxhyp0.1} and \ref{xxhyp0.2} for $A$. Then $A$ has a quasi-Frobenius artinian
quotient ring.
\end{theorem}

\begin{theorem}
\label{xxthm0.4}
Assume Hypotheses~\ref{xxhyp0.1} and \ref{xxhyp0.2} for $A$. If $A$ has finite global dimension, then 
$A$ is a direct sum of prime algebras and each summand is homogeneous,
Artin--Schelter regular, Auslander regular, and Cohen--Macaulay.
\end{theorem}

The following is a version of the Nichols--Zoeller Theorem for infinite-dimensional algebras.
\begin{theorem}
\label{xxthm0.5}
Assume Hypotheses~\ref{xxhyp0.1} and \ref{xxhyp0.2} for both $A_1$ and $A_2$. Let $A_1$ 
and $A_2$ be homogeneous of the same Gelfand--Kirillov dimension. 
Suppose that $A_1$ has finite global dimension. If there is an algebra map 
(which is not necessarily a coalgebra map) $f: A_1\to A_2$ 
such that $A_2$ is a finitely generated module over $A_1$ on both 
sides, then $A_2$ is a projective module over $A_1$ on both sides.
\end{theorem}

\begin{theorem}
\label{xxthm0.6}
Assume Hypotheses~\ref{xxhyp0.1} and \ref{xxhyp0.2} for $A$.
\begin{enumerate}
\item[(1)]
$A$ has a rigid Auslander and Cohen--Macaulay dualizing complex
which is also an invertible complex of $A$-bimodules.
\item[(2)]
$A$ has a residue complex.
\item[(3)]
If $A$ is homogeneous, then $A$ has a minimal pure injective 
resolution on both sides.
\end{enumerate}
\end{theorem}

The following result concerns the homological properties of 
$A$-modules.

\begin{theorem}
\label{xxthm0.7}
Assume Hypotheses~\ref{xxhyp0.1} and \ref{xxhyp0.2} for $A$ and let $d=\GKdim A$.
Let $M$ be a nonzero finitely generated left 
$A$-module.
\begin{enumerate}
\item[(1)]{\rm(The Auslander--Buchsbaum formula)}
If $\projdim M<\infty$, then
\[
\projdim M\leq d-\depth M.
\]
If, further, $A$ is homogeneous, then
\[
\projdim M= d-\depth M.
\]
\item[(2)]{\rm (Bass's theorem)}
If $\injdim M<\infty$, then
\[
\injdim M\leq d.
\]
If, further, $A$ is homogeneous, then
\[
\injdim M=d.
\]
\item[(3)]{\rm(The no-holes theorem)} 
Suppose either $A$ is homogeneous or $M$ is 
indecomposable.
For every integer $i$ between $\depth M$ and 
$\injdim M$,
there is a simple left 
$A$-module $S$ such that $\Ext^i_A(S,M)\neq 0$.
\end{enumerate}
\end{theorem}

We begin, in Section~\ref{xxsec1}, by reviewing the definitions of various homological properties and introducing, for a noetherian algebra $A$, important homological conditions (L1) and (R1) on the categories of left and right $A$-modules, respectively. 
In Section~\ref{xxsec2}, we prove some key lemmas which allow us, in Section~\ref{xxsec3}, to prove that if $A$ is a finite module over its affine center and satisfies (L1) and (R1), then $A$ is a finite direct sum of AS Gorenstein, Auslander Gorenstein, and Cohen--Macaulay algebras.
We study the consequences of this result in Section~\ref{xxsec4} and show that the conclusions of Theorems~\ref{xxthm0.3}--\ref{xxthm0.7} hold for $A$.
In Section~\ref{xxsec5}, we recall the definition of a weak Hopf algebra and provide some examples.
In Section~\ref{sec:tc}, we study the relationship between Hypothesis~\ref{xxhyp0.2} and the conditions (L1) and (R1) and conclude in Section~\ref{xxsec7} that weak Hopf algebras which are module-finite over their affine centers satisfy (L1) and (R1). We then prove Theorems~\ref{xxthm0.2}--\ref{xxthm0.7}, settling the Brown--Goodearl Question for this class of weak Hopf algebras.
We conclude in Section~\ref{xxsec8} by posing some open questions.

\subsection*{Acknowledgments}
R. Won was partially supported by an AMS--Simons Travel Grant.
J.J. Zhang was partially supported by the US National Science 
Foundation (No. DMS-1700825).

\section{preliminaries}
\label{xxsec1}

We first recall some definitions concerning different homological 
properties. For an algebra $A$, let $A\Modleft$ denote the category of left $A$-modules; 
for $M, N \in A\Modleft$ we write $\Hom_A(M,N)$ for the space of left module 
homomorphisms. We identify the category $\Modright A$ of right $A$-modules 
with $A^{op}\Modleft$ when convenient. In particular, for right modules $M$ 
and $N$ we write $\Hom_{A^{op}}(M, N)$ for the space 
of right module homomorphisms.

\begin{definition}\cite[Definitions 1.2 and 2.1]{Le}
\label{xxdef1.1}
Let $A$ be an algebra and $M$ a left $A$-module.
\begin{enumerate}
\item[(1)]
The {\it grade number} of $M$ is defined to be
\[
j_A(M):=\inf\{i \mid \Ext_A^i(M,A)\neq0\}\in {\mathbb N}\cup\{+\infty\}.
\]
We often write $j(M)$ for $j_A(M)$. Note that 
$j_A(0)=+\infty$.  
\item[(2)] 
We say that $M$ satisfies the {\it Auslander condition} if for any $q\geq0,$
$j_A(N)\geq q$ for all right $A$-submodules $N$ of $\Ext_A^q(M,A)$.
\item[(3)] 
We say a noetherian algebra $A$ is {\it Auslander Gorenstein} (respectively, 
{\it Auslander regular}) of dimension $n$ if 
$\injdim A_A=\injdim {_AA}=n<\infty$ (respectively, $\gldim A=n<\infty$),
and every finitely generated left and right $A$-module satisfies 
the Auslander condition.
\end{enumerate}
\end{definition}

\begin{definition}\cite[Definition 0.4]{ASZ1}
\label{xxdef1.2}
We say $A$ is {\it Cohen--Macaulay} (or, {\it CM} 
for short) if $\GKdim(A)=d<\infty $ and 
\[
j(M)+\GKdim(M)=\GKdim(A)
\]
for every finitely generated nonzero left (or right) $A$-module $M$. 
\end{definition}

In this paper we will use the following slightly 
modified version of the Artin--Schelter Gorenstein
property defined in \cite[Definition 3.1]{WZ1}. 

\begin{definition}
\label{xxdef1.3} 
A noetherian algebra $A$ is called {\it Artin--Schelter Gorenstein} 
(or \emph{AS Gorenstein}, for short) if the following conditions hold:
\begin{enumerate}
\item[(1)]
$A$ has finite injective dimension $d<\infty$ on both sides.
\item[(2)]
For every finite-dimensional left $A$-module $S$,
$\Ext^i_A(S, A)=0$ for all $i \neq d$ and $\dim\Ext^d_{A} (S, A) 
<\infty$.
\item[(3)]
The analog of part (2) for right $A$-modules holds.
\end{enumerate}
If, moreover,
\begin{enumerate}
\item[(4)]
$A$ has finite global dimension,
\end{enumerate}
then $A$ is called 
{\it Artin--Schelter regular} (or \emph{AS regular}, for short).
\end{definition}

Let $A \Modleftfd$ denote the category of finite-dimensional left $A$-modules. The category of finite-dimensional right $A$-modules will be written $\Modrightfd A$ 
or $A^{op} \Modleftfd$.
\begin{lemma}
\label{xxlem1.4}
Suppose $A$ is AS Gorenstein of injective dimension $d$. Then 
$A \Modleftfd$ is contravariant equivalent to $A^{op} \Modleftfd$ via 
the functor $\Ext^d_{A}(-,A)$. As a consequence, for each $i$, both 
$\Ext^i_{A}(-,A)$ and $\Ext^i_{A^{op}}(-,A)$ are exact functors
when restricted to finite-dimensional $A$-modules. 
\end{lemma}

\begin{proof}
By Definition~\ref{xxdef1.3}(2), $\Ext^d_{A}(-,A)$ is an exact
functor on finite-dimensional left $A$-modules. For every
finite-dimensional left $A$-module $S$, by Ischebeck's double 
Ext-spectral sequence \cite[(0-2)]{ASZ1}, we have
\[
\Ext^d_{A^{op}}(\Ext^d_{A}(S,A),A)\cong S.
\]
A similar statement holds for finite-dimensional right $A$-modules
$S$. The assertion follows, and the consequence is clear.
\end{proof}
Lemma~\ref{xxlem1.4} implies that if $A$ is AS Gorenstein of dimension $d$ 
then $\Ext^d_A(S, A)$ is simple right $A$-module for each finite-dimensional simple left module $S$, and similarly 
on the other side. This shows that Definition~\ref{xxdef1.3} is equivalent to the definition 
of AS Gorenstein in \cite[Definition 3.1]{WZ1} for any algebra $A$ for which all simple modules 
are finite-dimensional. In particular, this is the case for the affine noetherian PI algebras 
of main interest in this paper by \cite[Proposition 3.1]{BG}.

\begin{definition}
\label{xxdef1.5} 
Let $A$ be a noetherian algebra.
\begin{enumerate}
\item[(1)]
We say that $A$ satisfies (L1) (respectively, (R1)) if, for each 
$i\geq 0$, the functor $\Ext^i_A(-,A)$ is exact when applied to 
the category $A \Modleftfd$ (respectively, $\Ext^i_{A^{op}}(-, A)$ 
is exact when applied to $\Modrightfd A$).
\item[(2)]
We say that $A$ satisfies (L2) (respectively, (R2)) if $A$ satisfies
(L1) (respectively, (R1)), and for each 
$i\geq 0$, for $0 \neq S, T \in A \Modleftfd$, $\Ext^i_A(S,A)=0$ if and only if 
$\Ext^i_A(T,A)=0$ (respectively, for $0 \neq S, T \in \Modrightfd A$, $\Ext^i_{A^{op}}(S,A)=0$ if and only if 
$\Ext^i_{A^{op}}(T,A)=0$).
\end{enumerate}
\end{definition}

By \cite[Proposition 3.2]{WZ1}, if $A$ is affine noetherian PI, then 
$A$ is AS Gorenstein if and only if $A$ satisfies (L2) (or 
equivalently, satisfies (R2)). 
The following lemma follows 
from algebra decomposition and \cite[Proposition 3.2]{WZ1}. 
\begin{lemma}
\label{xxlem1.6}
Suppose that $A$ is an affine noetherian PI algebra. If $A$ is a direct
sum of finitely many AS Gorenstein algebras (of possibly different dimensions), then $A$ satisfies {\rm (L1)}
and {\rm (R1)}.
\end{lemma}
One of our main goals is to show that the converse of Lemma 
\ref{xxlem1.6} holds under some extra hypotheses.

We will need the following easy lemmas about algebra decompositions.
\begin{lemma}
\label{xxlem1.7} 
Let $A$ be an algebra.
\begin{enumerate}
\item[(1)]
Let $e$ and $e'$ be two idempotents
such that $eA=Ae'$. Then $e=e'$ is
a central idempotent.
\item[(2)]
Let $I$ be a two sided ideal of $A$ such that
$A=I\oplus B$ as left $A$-modules and $A=I\oplus C$
as right $A$-modules. Then $B=C$ and $A=I\oplus B$ 
as algebras.
\end{enumerate}
\end{lemma}

\begin{proof} (1) Since $eA=Ae'$,
there are elements $a,b\in A$ such that
$e=ae'$ and $e'=eb$. Then 
\[
e=ae'=ae' e'=ee'=e eb=eb=e'.
\]
For every $a\in A$, $ea, ae\in eA=Ae$.
Hence $ea=eae=ae$. This shows that
$e$ is a central idempotent.

(2) Since $A=I\oplus B$, $1=e+(1-e)$ 
where $e\in I$ and $(1-e)\in B$. Since
$e(1-e)\in I\cap B=0$, therefore $e=e^2$ and so $e$ is 
idempotent. Since $Ae\oplus A(1-e)=A=I\oplus B$,
we obtain that $Ae=I$ and $A(1-e)=B$. 

Similarly, there is an idempotent $e'$ such that
$I=e'A$. By part (1), $e=e'$ which is central.
Therefore $A=eA\oplus (1-e)A$ where both
$eA$ and $(1-e)A$ are two-sided ideals of $A$.
The assertion follows.
\end{proof}

We say that $A$ is \emph{indecomposable} if $A$ is not 
isomorphic to a direct sum of two algebras.
The following lemma is standard.

\begin{lemma} \cite[Proposition 22.2]{Lam}
\label{xxlem1.8}
Let $A$ be a noetherian algebra. Then 
\[
A=\bigoplus_{i\in I} A_i
\]
for a finite set of indecomposable algebras
$\{A_i\}_{i\in I}$. Further, this decomposition 
is unique up to permutation.
\end{lemma}

\section{Some key lemmas}
\label{xxsec2}

While we primarily consider the GK dimension of modules as our dimension 
function in this paper, in this section it is convenient to use also the (Gabriel-Rentschler) Krull dimension of a module $M$, 
which we denote by $\Kdim M$. Fortunately, if $A$ is finite over its affine center, 
then by \cite[Lemma 1.2(3)]{WZ1}, for all finitely generated left (or right) $A$-modules $M$, 
\[\Kdim M=\GKdim M.\]
For any such module $M$, for every $s$, 
let $\tau_s(M)$ denote the largest submodule of $M$ with GK dimension which is less than or equal to $s$.
If $M$ is an $(A, A)$-bimodule which is finitely generated
on both sides, then 
\[\tau_s(_AM)=\tau_s(M_A)
\]
since GK dimension is symmetric \cite[Corollary 5.4]{KL}. In particular,
\[\tau_s(_AA)=\tau_s(A_A)
\]
which we denote by $\tau_s(A)$. 

Let us introduce some temporary notation. Let $A\modleft$
denote the category of finitely generated left $A$-modules. For
each integer $d$, let $A\modleft_{d}$ denote the category
of finitely generated left $A$-modules of Krull dimension no 
more than $d$. The following is a key lemma.

\begin{lemma}
\label{xxlem2.1}
Let $A$ be a noetherian PI complete semilocal algebra and $M$
be a finitely generated left $A$-module. If $\Hom_A(M,-)$
is exact when applied to $A\modleft_{0}$, then $M$ is projective.
\end{lemma}

\begin{proof}
Since $A$ is noetherian and complete semilocal, it is semiperfect 
in the sense of \cite[Definition 23.1]{Lam}. Then by \cite[Proposition 24.12]{Lam}, every finitely 
generated $A$-module has a projective cover. 
Let $P$ be the projective cover of $M$ and let $K$ be the kernel of 
the surjective map $P\to M$. Then we have a short exact sequence
\[
0\to K\xrightarrow{f} P\xrightarrow{g} M\to 0.
\]
It suffices to show that $K=0$. 

Let ${\mathfrak m}$ be the Jacobson radical of $A$. Then a 
finitely generated left $A$-module $N$ has finite
length if and only if ${\mathfrak m}^s N=0$ for some integer 
$s$. Let $n$ be any positive integer. By adjunction, 
\[\Hom_{A/{\mathfrak m}^n}(A/{\mathfrak m}^n\otimes_A M,-)
\cong \Hom_{A}(M,-)
\]
when applied to modules over $A/{\mathfrak m}^n$. Hence
$\Hom_{A/{\mathfrak m}^n}(A/{\mathfrak m}^n\otimes_A M,-)$
is exact when applied to modules in $A/{\mathfrak m}^n\modleft_{0} = A/{\mathfrak m}^n\modleft$.
Since $A/{\mathfrak m}^n\otimes_A M$ is a finitely generated $A/{\mathfrak m}^n$-module, 
it follows that $A/{\mathfrak m}^n\otimes_A M$ 
is a projective module over $A/{\mathfrak m}^n$. 

It is easy to check that $A/{\mathfrak m}^n\otimes_A P$ 
is a projective cover of $A/{\mathfrak m}^n\otimes_A M$.
Then 
\[
A/{\mathfrak m}^n\otimes_A g:\qquad A/{\mathfrak m}^n\otimes_A P
\to A/{\mathfrak m}^n\otimes_A M
\]
is an isomorphism. Therefore 
\[
A/{\mathfrak m}^n\otimes_A f:\qquad A/{\mathfrak m}^n\otimes_A K
\to A/{\mathfrak m}^n\otimes_A P
\]
is the zero map for all $n\geq 1$. Equivalently,
$f(K)\subseteq {\mathfrak m}^n P$ for all $n$. 
Since $A$ is PI, by \cite[Theorem 9.13]{GW}, 
$\bigcap_{n} {\mathfrak m}^n=0$. Hence $\bigcap_{n} {\mathfrak m}^n P=0$, 
as $P$ is a finitely generated projective $A$-module, 
and consequently, $f(K)=0$. Since $f$ is monomorphism, 
$K=0$ as required. 
\end{proof}

Assume that $A$ is finitely generated over its affine center $Z(A)$.
Let ${\mathfrak n}$ be a maximal ideal 
of $Z(A)$. Let $Z_{\mathfrak n}$ denote the completion of the 
commutative noetherian local ring $Z(A)_{\mathfrak n}$ with 
respect to its maximal ideal. Then 
\begin{enumerate}
\item[(1)]
$Z_{\mathfrak n}$ is noetherian \cite[Theorem 10.26]{AM}.
\item[(2)]
$A_{\mathfrak n}:=Z_{\mathfrak n}\otimes_{Z(A)} A$ is finitely generated
over $Z_{\mathfrak n}$ (but its center could be bigger
than $Z_{\mathfrak n}$).
\item[(3)]
$A_{\mathfrak n}$ is complete semilocal.
\item[(4)]
The functor $Z_{\mathfrak n}\otimes_{Z(A)}-: {A\modleft}
\to {A_{\mathfrak n}\modleft}$ is exact \cite[Proposition 10.14]{AM}.
\end{enumerate}

Every left $A$-module $M$ can be considered as an $(A,Z(A))$-bimodule.
It is well-known that if $M$ and $W$ are left $A$-modules, then $\Ext^i_A(M,W)$ 
has a central $(Z(A), Z(A))$-bimodule structure. 

\begin{lemma}
\label{xxlem2.2} 
Let $A$ be a finitely generated module over its affine center.
Suppose $M$ is a finitely generated left $A$-module. Then $M$ is 
projective over $A$ if and only if 
\[
M\otimes_{Z(A)} Z_{\mathfrak n} \cong Z_{\mathfrak n}\otimes_{Z(A)} M 
\]
is projective over $A_{\mathfrak n}$ for all maximal ideals 
${\mathfrak n}$ of $Z(A)$. 
\end{lemma}

\begin{proof} Since $M \otimes_{Z(A)} Z_{\mathfrak{n}} \cong M \otimes_A A_{\mathfrak{n}}$, one implication is clear. For the other implication,
assume that $M\otimes_{Z(A)} Z_{\mathfrak n}$ is projective 
for all maximal ideals ${\mathfrak n}$ of $Z(A)$. If $M$ is not 
projective, then there is a finitely generated $A$-module $W$ 
such that $\Ext^i_A(M,W)\neq 0$ for some $i>0$. Let  
${\mathfrak n}$ be a maximal ideal of $Z(A)$ such that 
\[
\Ext^i_A(M,W)\otimes_{Z(A)} Z_{\mathfrak n}\neq 0.
\]
By \cite[Lemma 3.7]{YZ3},
\[
\Ext^i_{A_{\mathfrak n}}(M\otimes_{Z(A)} Z_{\mathfrak n},
W\otimes_{Z(A)} Z_{\mathfrak n})
\cong
\Ext^i_A(M,W)\otimes_{Z(A)} Z_{\mathfrak n}\neq 0.
\]
Hence $M\otimes_{Z(A)} Z_{\mathfrak n}$ is not projective,
a contradiction.
\end{proof}

We will use the following result in the analysis of the 
dualizing complex over $A$.

\begin{proposition}
\label{xxpro2.3}
Let $A$ be a finitely generated module over its affine center.
Suppose $M$ is a finitely generated left $A$-module such that
$\Hom_A(M,-)$ is exact on finite-dimensional left $A$-modules. 
Then $M$ is projective.
\end{proposition}

\begin{proof} By Lemma~\ref{xxlem2.2}, it suffices to show that
$M_{\mathfrak n}:=M\otimes_{Z(A)} Z_{\mathfrak n}$ is projective 
for all maximal ideals ${\mathfrak n}$ of $Z(A)$. 

Let $W$ be a finite-dimensional left $A_{\mathfrak n}$-module. Then 
\[
W\cong A_{\mathfrak n}\otimes_A W\cong 
W\otimes_{Z(A)} Z_{\mathfrak n}.
\]
By \cite[Lemma 3.7]{YZ3},
\[
\begin{aligned}
\Hom_{A_{\mathfrak n}}(M_{\mathfrak n}, W)
&\cong \Hom_{A}(M, W).
\end{aligned}
\]
Since $Z(A)/{\mathfrak n}$ is finite-dimensional, 
every finitely generated artinian module over $A_{\mathfrak n}$ is 
finite-dimensional. Hence, by hypothesis 
$\Hom_{A_{\mathfrak n}}(M_{\mathfrak n}, -) \cong \Hom_A(M, -)$ is exact when 
applied to objects in $A_{\mathfrak n}\modleft_0$. 
By Lemma~\ref{xxlem2.1}, $M_{\mathfrak n}$ is projective.
\end{proof}

We need one more homological lemma, which depends on the basic properties of 
Krull dimension for PI algebras.
\begin{lemma}
\label{xxlem2.4}
Let $A$ be a noetherian PI algebra of finite Krull dimension.
Let $M$ be an $(A, A)$-bimodule which is finitely generated on both sides
and let $w$ be a nonnegative integer.
\begin{enumerate}
\item[(1)]
Suppose that for all simple left $A$-modules $S$, we have $\Ext^s_A(S,M) = 0$ for all $s \leq w$.
Then, for each integer $d\geq 0$, if $N$ is a finitely generated left $A$-module with $\Kdim N \leq d$, we have $\Ext^s_A(N,M) = 0$ for all $s \leq w -d$.
As a consequence, $_AM$ does not contain
any nonzero left $A$-submodules of Krull dimension less than or equal to $w$.
\item[(2)]
Suppose that for all simple left $A$-modules $S$, we have $\Ext^s_A(S,M) = 0$ for all $s > w$.
Then for all finitely generated left $A$-modules $N$, we have $\Ext^s_A(N,M) = 0 $ for all $s > w$.
As a consequence, $\injdim (_AM)\leq w$.
\end{enumerate}
\end{lemma}

\begin{proof} (1) We prove the assertion by induction on $d$.
When $d=0$, it follows from exact sequences and the hypothesis
that $\Ext^s_A(S,M)=0$ for all finite length left $A$-modules $S$ and 
for all $s\leq w$. Since the $A$-modules of finite length are precisely 
the $A$-modules of Krull dimension $0$, the result holds for $d = 0$.

Now let $d>0$ and assume that the assertion 
$\Ext^s_A(N',M)=0$ holds for all finitely generated left $A$-modules $N'$
with $\Kdim N' \leq d-1$ and for all $s\leq w-(d-1)$.  We wish 
to show that $\Ext^s_A(N, M) = 0$ for all finitely generated left 
$A$-modules $N$ of Krull dimension $d$ and for all $s \leq w-d$.  
By choosing a prime filtration of $N$, similarly as in \cite[Lemma 2.1(i,ii)]{SZ1}, we may assume that $N=A/{\mathfrak p}=:B$ for some prime ideal ${\mathfrak p}$, and where $\Kdim(B) = d$.  If $x$ is a nonzero
central element $x\in B$, then $x$ is regular
and there is a short exact sequence
\[
0\to B \xrightarrow{r_x} B\to B/(x)\to 0,
\]
where $r_x$ denotes right multiplication by $x$ and
$B/(x)$ is a left $A$-module with $\Kdim(B/(x)) \leq d-1$.
By the induction hypothesis, $\Ext^s_A(B/(x), M)=0$
for all $s\leq w-d+1$. Then, for every $s \leq w-d$, the 
long exact sequence
\begin{align*}\dots \to &\Ext^{s}_A(B/(x),M)\to 
\Ext^{s}_A(B,M)\xrightarrow{(r_x)^*} 
\Ext^{s}_A(B,M) \to \\
& \Ext^{s+1}_A(B/(x),M)\to \dots
\end{align*}
implies that $(r_x)^*: \Ext^{s}_A(B,M)\to \Ext^{s}_A(B,M)$ is 
an isomorphism. Note that $\Ext^{s}_A(B,M)$ is a 
left $B$-module, and that $(r_x)^\ast$ is just left multiplication
$l_x$ by $x$ \cite[Lemma 3.4]{SZ1}. Let $Q(B)$ be the total 
fraction ring of $B$, 
obtained by inverting all central nonzero elements $x$. Since 
$l_x=(r_x)^{\ast}$ is an isomorphism, we can naturally define a left 
$Q(B)$-action on $\Ext^{s}_A(B,M)$ by setting $l_{x^{-1}}=(l_x)^{-1}$.
So $W:=\Ext^{s}_A(B,M)$ is a $(Q(B),A)$-bimodule.
By \cite[Theorem 3.5]{SZ1}, $W$ is finitely generated as a left $B$-module.  
Since $A$ is noetherian and $M$ is a noetherian right $A$-module, computing $W = \Ext^s_{A}(B, M)$ with a projective resolution of $B$ by finite rank free $A$-modules shows 
that $W$ is also a finitely generated right $A$-module.
Thus $W$ is a $(B,A)$-bimodule, finitely generated on both sides (as well as a $(Q(B),A)$-bimodule, also finitely generated on both sides). 
By Krull symmetry \cite[Theorem 15.15]{GW}, 
\[
\Kdim {_{B}W}=\Kdim W_A=\Kdim {_{Q(B)} W}=0.
\]
Since $\Kdim B=d>0$, there is a nonzero ideal $I$ of $B$ such that
$IW=0$. Since any nonzero ideal in a prime PI ring
contains a nonzero central element, $Q(B)=Q(B)I$. 
Therefore $W=Q(B)W=Q(B)IW=0$, as desired.

The consequence follows by taking $d=w$ and $s=0$.

(2) The assertion follows by induction on $d:=\Kdim N$. The proof is similar to the proof of (1), so it is omitted.                                                                                                                       
\end{proof}

\section{Dualizing complexes and residue complexes}
\label{xxsec3}


The noncommutative version of a dualizing complex was introduced in 
1992 by Yekutieli \cite{Ye1}. Let $\mb{D}^b_{f.g.}(A\Modleft)$ denote the 
bounded derived category of complexes of left $A$-modules with 
finitely generated cohomology modules.  Roughly speaking, a dualizing complex 
over an algebra $A$ is a complex $R$ of $A$-bimodules, such that 
the two derived functors ${\text{RHom}}_A(-,R)$ and 
${\text{RHom}}_{A^{op}} (-,R)$
induce a duality between the derived categories ${\mathbb{D}}^b_{f.g.}(A\Modleft)$ 
and ${\mathbb{D}}^b_{f.g.}(A^{op}\Modleft)$. Let $A^e = A \otimes_{\kk} A^{op}$ denote the 
enveloping algebra of $A$.

\begin{definition}
\label{xxdef3.1} Let $A$ be a noetherian algebra. A complex 
$R\in {\mathbb D}^b(A^{e}\Modleft)$ is called a {\it dualizing 
complex} over $A$ if it satisfies the three conditions below:
\begin{enumerate}
\item[(i)]
$R$ has finite injective dimension on both sides.
\item[(ii)]
$R$ has finitely generated cohomology modules on both sides.
\item[(iii)]
The canonical morphisms $A\to {\text{RHom}}_{A}(R,R)$ and 
$A\to {\text{RHom}}_{A^{op}}(R,R)$ in ${\mathbb D}(A^{e}\Modleft)$
are both isomorphisms.
\end{enumerate}
\end{definition}

We also need a few other definitions related to dualizing complexes.

\begin{definition}
\label{xxdef3.2} Let $A$ be a noetherian algebra and $R$ a 
dualizing complex over $A$.
\begin{enumerate}
\item[(1)]
Let $M$ be a finitely generated left $A$-module. The \emph{grade} of 
$M$ (with respect to $R$) is defined to be
\[
j_R(M)=\min\{ i \mid \Ext^i_A(M,R)\neq 0\}
\in {\mathbb Z}\cup\{\pm \infty\}.
\]
We write the grade of $M$ as $j(M)$ when the choice of $R$ is clear.
The grade of a right $A$-module is defined similarly.
\item[(2)] \cite[Definition 2.1]{YZ4}
We say that $R$ has the \emph{Auslander property}, or that $R$ is an
\emph{Auslander dualizing complex}, if 
\begin{enumerate}
\item[(i)]
for every finitely generated 
left $A$-module $M$, integer $q$, and right $A$-submodule
$N\subseteq \Ext^q_{A}(M, R)$, one has $j(N) \geq q$; 
\item[(ii)]
the same holds after exchanging left and right.
\end{enumerate}
\item[(3)] \cite[Definition 8.1]{VdB}
A dualizing complex $R$ is called {\it rigid} if there is an isomorphism
\[
R \xrightarrow{\cong} {\text{RHom}}_{A^e}(A, R\otimes R)
\]
in ${\mathbb D}(A^{e}\Modleft)$.
\item[(4)] \cite[Definition 2.24]{YZ4}
Suppose $R$ is an Auslander dualizing complex over $A$.
We say $R$ is {\it Cohen--Macaulay} if for every 
finitely generated left (respectively, right) $A$-module $M$,
\[
j(M)+\GKdim M=0.
\]
(This is a slightly stronger version than \cite[Definition 2.24]{YZ4}.)
\item[(5)] 
Suppose that $A$ has finite injective dimension $d$ as  
a left and right $A$-module, and that $\GKdim(A) = d$.  
Then the complex $R = A[d]$ is a dualizing complex for $A$, and we say that $A$ is \emph{Auslander Gorenstein} if $R$ is an Auslander dualizing complex, and that $A$ is \emph{Cohen--Macaulay} if $R$ is a Cohen--Macaulay dualizing complex.
\end{enumerate}
\end{definition}

Since we are working with algebras that are finite over their affine 
centers, the natural dimension function  to use is Gelfand--Kirillov 
dimension. Recall that an algebra $A$ is called (left) \emph{homogeneous} if $\GKdim L = \GKdim A$ for all nonzero left ideals $L \subseteq A$. 
This notion generalizes to $A$-modules $M$ as defined below.

\begin{definition}
\label{xxdef3.3} Let $A$ be a noetherian algebra and $M$ be a nonzero left 
$A$-module.
\begin{enumerate}
\item[(1)] \cite[Definition 0.2]{ASZ1}
Suppose $\GKdim M=s$. We say $M$ is {\it $s$-pure}, if 
$\GKdim N=s$ for all nonzero submodules $N$ of $M$. 
\item[(2)] \cite[Definition 0.3]{ASZ1}
Suppose $\injdim _{A}A=d<\infty$ and let
\[
0\to A\to I^{0}\to I^{1}\to \cdots \to I^{d}\to 0
\]
be a minimal injective resolution of the left $A$-module $_{A}A$.
We say this resolution is \emph{pure} if each $I^i$ is $(d-i)$-pure.
\end{enumerate}
\end{definition}

\begin{definition} \cite[Definitions 4.3 and 5.1]{YZ3}
\label{xxdef3.4} 
A dualizing complex $K$ over $A$ is called a {\it residual complex} 
over $A$ if the following conditions are satisfied:
\begin{enumerate}
\item[(i)] 
$K$ is Auslander,
\item[(ii)]
each $K^{-q}$ is an injective module over $A$ on both sides,
\item[(iii)]
each $K^{-q}$ is $q$-pure on both sides.
\end{enumerate}
A rigid residual complex $K$ is called a {\it residue} complex. 
\end{definition}

When a residue complex $K$ exists, then $_AK$ (respectively, $K_A$) 
can be viewed as a minimal injective resolution of a rigid dualizing 
complex $_AR$ (respectively, $R_A$).

Here are some basic facts about dualizing complexes for the algebras $A$ 
of interest in this paper.

\begin{lemma}
\label{xxlem3.5} 
Let $A$ be a finitely generated module over its affine center.
Then the following hold.
\begin{enumerate}
\item[(1)]
There is a rigid dualizing complex $R$ over $A$ that is Auslander 
and Cohen--Macaulay.
\item[(2)]
There is a residue complex $K$ over $A$ that is Auslander 
and Cohen--Macaulay.
\item[(3)]
Let $M$ be a finitely generated left $A$-module. If $\Ext^i_A(M,K)$ is nonzero
only when $i=-d$, then $\GKdim M=d$. 
\end{enumerate}
\end{lemma}

\begin{proof} 
(1) If $A$ is a finitely generated 
module over its affine center then it admits a noetherian 
connected filtration (see \cite[Remark 6.4]{YZ3}). Hence 
\cite[Proposition 6.5]{YZ3} applies, and shows that $A$ has an Auslander 
rigid dualizing complex $R$.  For a module $M$, the canonical dimension of $M$ 
is defined in \cite{YZ3} to be $\operatorname{Cdim}(M) = -j_R(M)$.  The proof 
of \cite[Proposition 6.5]{YZ3} shows that
for finitely generated left and right modules, the canonical dimension is equal to the GK dimension.
This implies that $R$ is Cohen--Macaulay.

(2) $A$ has a residue complex $K$ by \cite[Proposition 6.6]{YZ3}.  Since a residue complex 
is rigid, and a rigid dualizing complex is unique up to isomorphism in the derived category 
\cite[Proposition 8.2]{VdB}, $K$ is also Auslander and Cohen-Macaulay by part (1).

(3) By the double-Ext spectral sequence \cite[Proposition 1.7]{YZ4},
\[
M\cong \Ext^{-d}_{A^{op}}(\Ext^{-d}_A(M,K),K).
\]
By the Auslander property of $K$, $j(M)\geq -d$. By the Cohen--Macaulay 
property of $K$, $\GKdim M=-j(M)\leq d$. It remains to show that 
$\GKdim M\geq d$. If not, then by the purity of $K$, $\Hom_A(M,K^{-d})=0$, 
which implies that $\Ext^{-d}_A(M,K)= 0$, which is a contradiction. The assertion 
follows.
\end{proof}

Recall that an algebra $A$ is called indecomposable if it is not possible to write $A = B \oplus C$ as a direct sum of algebras.
The next result contains the bulk of the work needed for the proof of our main theorem.
\begin{theorem}
\label{xxthm3.6} 
Let $A$ be a finite module over its affine center. Suppose $A$ 
is indecomposable and satisfies {{\rm (L1)}} and {{\rm (R1)}} of 
Definition~\ref{xxdef1.5}(1). Then the following hold:
\begin{enumerate}
\item[(1)]
$A$ is AS Gorenstein of injective dimension $\GKdim A$; in particular, 
$A$ satisfies conditions {\rm (L2)} and {\rm (R2)} of 
Definition~\ref{xxdef1.5}(2).
\item[(2)]
$A$ is Auslander Gorenstein and Cohen--Macaulay.
\item[(3)]
$A$ is a homogeneous $A$-module on both sides.
\end{enumerate}
\end{theorem}

\begin{proof} 
Part (2) follows from part (1) and \cite[Theorem 1.3]{SZ1}.  Part (3) 
follows from the Cohen--Macaulay property of $A$.  Hence, we only need to 
prove part (1).  

By Lemma~\ref{xxlem3.5}(2), $A$ has a residue 
complex $K$ which is Auslander and Cohen--Macaulay.  Let $d = \GKdim A$.  
Note that by the definition 
of residue complex, we must have $K^{-i} = 0$ for $i < 0$ and $i > d$.
 
First, we claim that for all $i \neq d$, $H^{-i}(K)=0$. 
Suppose this claim is not true. 
Then there is $0 \leq s<d$ such that $H^{-s}(K)\neq 0$. Choose
the smallest such $s$. Let $\Omega$ be the nonzero $(A, A)
$-bimodule 
$H^{-s}(K)$. By the definition of a dualizing complex, $\Omega$ is finitely generated over $A$ on both sides. Let 
\[D:={\text{RHom}}_A(-,K): {\mathbb D}^b_{f.g.}(A\Modleft)\to 
{\mathbb D}^b_{f.g}(\Modright A)
\]
and
\[
D^{op}:= {\text{RHom}}_{A^{op}}(-,K): {\mathbb D}^b_{f.g.}(\Modright A)
\to {\mathbb D}^b_{f.g.}(A\Modleft)
\]
be the duality functors in \cite[Proposition 3.4]{Ye1}.  Recall 
from the definition that each term $K^{-i}$ is an injective, $i$-pure module 
on both the left and the right.  It follows 
that if $S$ is a finite-dimensional left 
$A$-module then 
\[
H^i(D(S))= H^i(\RHom_A(S, K)) = H^i(\Hom_A(S, K)) = \begin{cases} 0 & i\neq 0,\\ 
\Hom_A(S, K^0) & i=0,\end{cases}
\]
and $H^0(D(S))$ is finitely generated on the right. Since $K$ is Auslander and Cohen--Macaulay,
it follows that $H^0(D(S))$ is finite-dimensional.  Since the complex $D(S)$ has cohomology 
in only one degree, we have $D(S) \cong T$ in ${\mathbb D}^b_{f.g.}(\Modright A)$ for 
some finite-dimensional right module $T = \Hom_A(S, K^0) \in \Modright A$.  We have 
$D^{op}(D(S)) \cong S$ in ${\mathbb D}^b_{f.g.}(A \Modleft)$ \cite[Proposition 4.2(2)]{Ye2}, 
and hence also in $A \Modleft$.  A similar result holds on the other side, and we 
conclude that $D$ and $D^{op}$ induce a duality between finite-dimensional left 
and right $A$-modules.  

Next, let $S$ be any left $A$-module. Let $P$ be a projective
resolution of $K$ (as a complex of left $A$-modules). Since 
$H^{-i}(K)=0$ for all $i\leq s$, we can assume that $P^{-s+i}=0$ for 
all $i>0$. Then $\Ext^i_A(K,S)\cong \Ext^i_{A}(P,S)=0$ for all
$i<s$. Using the fact that $\Hom_A(-,S)$ is left exact, one sees
that 
\[
\Ext^s_A(K,S)\cong \Ext^s_A(P,S)\cong \Hom_A(H^{-s}(P),S)
\cong \Hom_A(\Omega,S).
\]
Note that we have $A \cong D(D^{op}(A)) = D(K)$ in $\Modright A$ \cite[Proposition 4.2(1)]{Ye2}.
The dualizing functor $D$ also gives an isomorphism
\[
\RHom_A(M,N) \cong \RHom_{A^{op}}(D(N), D(M)),
\]
which is functorial in all $M, N \in {\mathbb D}^b_{f.g.}(\Modright A)$ \cite[Proposition 4.2(2)]{Ye2}.  Hence
\begin{align}
\label{eq:main}
\tag{E3.6.1}
\begin{split}
\Ext^i_{A^{op}}(D(S), A) &\cong \Ext^i_{A^{op}}(D(S), D(K)) \\
 &\cong \Ext^i_A(K,S)=\begin{cases} 0 & i<s\\ \Hom_{A}(\Omega, S)& i=s.
\end{cases}
\end{split}
\end{align}

As a consequence, if $T$ is any finite-dimensional right $A$-module, then since $T = D(S)$ 
where $S = D^{op}(T)$ is a finite-dimensional left module, we have 
$\Ext^i_{A^{op}}(T,A)=0$ for all $i<s$.  Now $D$, considered as a duality from finite-dimensional 
left $A$-modules to finite-dimensional right $A$-modules, is an exact functor.
Since the functor $\Ext^s_{A^{op}}(-,A)$ is exact on finite-dimensional right modules by hypothesis (R1), we see that $\Ext^s_{A^{op}}(D(-), A)$ is an exact functor from finite-dimensional left 
$A$-modules to left $A$-modules.  Now \eqref{eq:main} is functorial in $S$ 
and so $\Ext^s_{A^{op}}(D(-), A) \cong \Hom_A(\Omega,-)$ is exact on finite-dimensional left $A$-modules. By Proposition~\ref{xxpro2.3}, 
$\Omega$ is projective over $A$ as a left module.

Now because $\Omega\cong H^{-s}(K)$ is projective, where $H^{-i}(K) = 0$ for $i < s$, the complex $K$ is isomorphic to $K'\oplus \Omega[s]$ as a complex of left $A$-modules, where 
$H^{-i}(K')=0$ for all $i\leq s$. Since $K$ is a dualizing complex, 
we have, as complexes of right $A$-modules,
\[
\begin{aligned}
A&\cong {\text{RHom}}_{A}(K,K) \cong {\text{RHom}}_{A}(K'\oplus \Omega[s],K)\\
&\cong {\text{RHom}}_{A}(K',K)\oplus {\text{RHom}}_{A}(\Omega[s],K)\\
&\cong {\text{RHom}}_{A}(K',K)\oplus \Ext^{0}_A(\Omega[s], K)
\end{aligned}
\]
and $\Ext^{i}_A(\Omega[s], K)=0$ for all $i\neq 0$. The last 
assertion is equivalent to 
\[
\Ext^{i}_A(\Omega, K)=0
\]
for all $i\neq -s$. By Lemma~\ref{xxlem3.5}(3), $\GKdim \Omega=s$. 
So we have a right $A$-module decomposition
\[
A\cong A'\oplus V
\]
where $A'={\text{RHom}}_{A}(K',K)$ and $V \cong \Ext^{-s}_A(\Omega, K)$. 
By the Auslander and Cohen--Macaulay properties of $K$, $\GKdim V\leq s$.
Now from the equation $\Ext^i_A(K, S) \cong \Ext^i_A(K', S) \oplus \Ext^{i-s}_A(\Omega, S)$ and \eqref{eq:main}, we obtain for every finite-dimensional left $A$-module $S$
that $\Ext^i_A(K', S) = 0$ for all $i \leq s$.  Then for any finite dimensional right 
$A$-module $T$, writing $T \cong D(S)$ for $S = D^{op}(T)$, for all $i \leq s$ we obtain 
 \[
\Ext^i_{A^{op}}(T, A') = \Ext^{i}_{A^{op}}(D(S), A')\cong \Ext^i_{A^{op}}(D(S), D(K'))
\cong \Ext^i_A(K',S)= 0.
\]
By the right-sided version of Lemma~\ref{xxlem2.4}(1), $A'$ does not have a submodule of 
$\GKdim$ $\leq s$. Therefore $V=\tau_s(A)$. So we have a 
canonical decomposition
\[
A=A'\oplus \tau_s(A)
\]
as right $A$-modules. By symmetry, there is a decomposition 
\[
A=A''\oplus \tau_s(A)
\]
as left $A$-modules. By Lemma~\ref{xxlem1.7}(2), $A= B\oplus 
\tau_s(A)$ as algebras. This yields a contradiction to the hypothesis 
that $A$ is indecomposable.  This, finally, proves the claim that 
$H^{-i}(K) = 0$ for $i \neq d$.

We conclude that $K \cong \Omega[d]$ for some $(A, A)$-bimodule $\Omega$, which 
we have seen is projective as a left $A$-module.  
By repeating the above proof, we obtain that $\Omega$ is a 
projective $A$-module on both sides. This means that $K$ has 
finite projective dimension on both sides. Further, by \eqref{eq:main}, for every 
finite-dimensional left $A$-module $S$,
\[
\Ext^i_{A^{op}}(D(S),A)=\begin{cases} 0 & i<d\\ \Hom_{A}(\Omega, S)& i=d.
\end{cases}
\]
For every $i>d$,
\[
\Ext^i_{A^{op}}(D(S),A)\cong \Ext^i_{A}(K, S)\cong 
\Ext^i_{A}(\Omega[d], S) \cong \Ext^{i-d}_A(\Omega,S)=0
\]
as $\Omega$ is projective as a left $A$-module. Thus for every finite 
dimensional right $A$-module $T$, we have $\Ext^i_{A^{op}}(T, A) = 0$ unless 
$i =d$, in which case $\Ext^d(T, A)$ is a finite dimensional left $A$-module, as we have seen.
Also, by a right-sided version of Lemma \ref{xxlem2.4}(2), $A$ has finite injective dimension $d$ as a right module.  Symmetric arguments prove these facts on the other side.  Thus we have proved that $A$ is AS Gorenstein by definition.  That $A$ satisfies conditions (L2) and (R2) 
follows from Lemma~\ref{xxlem1.4}.
\end{proof}

\begin{corollary}
\label{xxcor3.7} 
Let $A$ be a finite module over its affine center. Suppose $A$ 
satisfies {{\rm (L1)}} and {{\rm (R1)}} of 
Definition~\ref{xxdef1.5}(1). Then $A$ is a finite direct sum $A = \bigoplus_{i \in I} A_i$
of homogeneous AS Gorenstein algebras of injective dimension 
$\leq \GKdim A$. Further, for each $i \in I$, 
$A_i$ satisfies the conclusions of Theorem~\ref{xxthm3.6}.
\end{corollary}

\begin{proof} By Lemma~\ref{xxlem1.8}, $A$ is a finite direct sum 
of indecomposable algebras, $A=\bigoplus_{i\in I} A_i$. It is easy
to show that each $A_i$ satisfies {{\rm (L1)}} and {{\rm (R1)}} of 
Definition~\ref{xxdef1.5}(1).  Since $Z(A) = \bigoplus_{i \in I} Z(A_i)$, it 
is also easy to see that each $A_i$ is finite over its affine center.
The assertion follows by applying 
Theorem~\ref{xxthm3.6} to each component $A_i$.
\end{proof}

\begin{corollary}
\label{xxcor3.8} 
Let $A$ be a finite module over its affine center. Suppose $A$ 
satisfies {{\rm (L1)}} and {{\rm (R1)}} of 
Definition~\ref{xxdef1.5}(1). Write $A=\bigoplus_{i\in I}A_i$ 
where each $A_i$ is indecomposable of GK dimension $d_i$. Then 
$A$ has a rigid dualizing complex $R$ of the form 
$\bigoplus_{i\in I} \Omega_i[d_i]$, where each $\Omega_i$ is an
invertible $(A_i, A_i)$-bimodule. As a consequence, $R$ is an invertible
complex of $(A, A)$-bimodules.
\end{corollary}

\begin{proof}
If $R =\bigoplus_{i\in I} \Omega_i[d_i]$ where $\Omega_i$ is 
an invertible $A_i$-bimodule, then $R$ has inverse 
$\bigoplus_{i\in I} \Omega_i^{-1}[-d_i]$. So we only need to 
prove the assertion when $A$ is indecomposable. Let $d=\GKdim A$.
By the proof of Theorem~\ref{xxthm3.6}, $A$ has a rigid dualizing
complex of the form $\Omega[d]$, where $\Omega$ is a projective 
on both sides.  We have also seen in Theorem~\ref{xxthm3.6} that 
$A$ has finite injective dimension $d$ as a left and right $A$-module.  
It is then obvious from the definition that $A$ is also a dualizing complex over $A$.
By \cite[Theorem 4.5(ii)]{Ye2}, $\Omega[d]={\text{RHom}}_{A}(A,\Omega[d])$
is a tilting complex over $A$, which is a (derived) invertible 
complex of $A$-bimodules \cite[Theorem 1.6]{Ye2}. Since
$\Omega$ is projective over $A$ on both sides, 
$\Omega$ must be an invertible $(A, A)$-bimodule.
\end{proof}

\section{Some consequences}
\label{xxsec4}

In this section we give some immediate consequences for algebras
that satisfy the conclusions of Corollary~\ref{xxcor3.7} or, in other words,  
the following hypothesis.

\begin{hypothesis}
\label{xxhyp4.1} 
Let $A$ be a finite direct sum $A = \bigoplus_{i \in I} A_i$ of indecomposable, noetherian, homogeneous, AS Gorenstein,
Auslander Gorenstein, Cohen--Macaulay algebras. For each 
$i \in I$, suppose $\injdim A_i = \GKdim A_i$.
\end{hypothesis}

\begin{lemma}
\label{xxlem4.2}
Suppose $A$ satisfies Hypothesis~\ref{xxhyp4.1}. Then $A$
has a quasi-Frobenius artinian quotient ring.
\end{lemma}

\begin{proof} It suffices
to show that for each $i \in I$, $A_i$ has a quasi-Frobenius artinian quotient ring.
This follows from \cite[Corollary 6.2]{ASZ1}.
\end{proof}

\begin{lemma}
\label{xxlem4.3}
Suppose $A = \bigoplus_{i \in I} A_i$ satisfies Hypothesis~\ref{xxhyp4.1}. If $A$ is a PI
algebra of finite global dimension, then each $A_i$ is a prime, homogeneous, AS regular, Auslander regular and Cohen--Macaulay
algebra.
\end{lemma}

\begin{proof} By definition, each $A_i$ is AS regular, 
Auslander regular and Cohen--Macaulay. By \cite[Theorem 5.4]{SZ1},
$A_i$ is prime. A prime PI algebra of finite GK dimension is homogeneous \cite[Lemma 10.18]{KL}. 
The assertion follows.
\end{proof}

We also recall a result in \cite{WZ1}.

\begin{lemma} \cite[Proposition 4.3]{WZ1}
\label{xxlem4.4}
Let $A$ and $B$ be noetherian, PI, AS Gorenstein algebras of the
same injective dimension and let $A\to B$ be an algebra homomorphism 
such that ${}_A B$ and $B_A$ are finitely generated. Then the following 
are equivalent:
\begin{enumerate}
\item[(1)]
$\injdim B_A <\infty$.
\item[(2)] 
$\injdim {_AB} <\infty$.
\item[(3)] 
$\projdim B_A < \infty$.
\item[(4)]
$\projdim {_AB} < \infty$.
\item[(5)] 
$B_A$ is projective over $A$.
\item[(6)] 
$_AB$ is projective over $A$.
\end{enumerate}
As a consequence, if $A$ has finite global dimension, 
then both $B_A$ and $_AB$ are projective.
\end{lemma}

We recall the following definition given in \cite[p. 1059]{WZ1}
or \cite[p. 521]{WZ2}.

\begin{definition}
\label{xxdef4.5}
Let $A$ be an algebra and $M$ be a left $A$-module. The 
{\it depth} of $M$ is defined to be
\[
\depth M:=\min \{i \mid \Ext^i_A(S,M)\neq 0 \;\;
{\text{for some simple $A$-module $S$}}\}.
\]
\end{definition}

\begin{proposition}
\label{xxpro4.6}
Let $A$ be a PI algebra satisfying Hypothesis~\ref{xxhyp4.1}. 
Let $d=\GKdim A$. Let $M$ be a nonzero finitely generated left 
(or right) $A$-module.
\begin{enumerate}
\item[(1)]{\rm(The Auslander--Buchsbaum formula)}
If $\projdim M<\infty$, then
\[
\projdim M\leq d-\depth M.
\]
If, further, $A$ is homogeneous, then
\[
\projdim M= d-\depth M.
\]
\item[(2)]{\rm(Bass's theorem)}
If $\injdim M<\infty$, then
\[
\injdim M\leq d.
\]
If, further, $A$ is homogeneous, then
\[
\injdim M=d.
\]
\item[(3)]{\rm(The no-holes theorem)} 
Suppose either $A$ is homogeneous or $M$ is indecomposable.
Then for every integer $i$ between $\depth M$ and 
$\injdim M$,
there is a simple left (or right) 
$A$-module $S$ such that $\Ext^i_A(S,M)\neq 0$.
\end{enumerate}
\end{proposition}

\begin{proof} First, if $A$ is homogeneous, then $A$ is 
AS Gorenstein. In this case all statements are given in 
\cite[Theorem 0.1]{WZ2}. 

Next, we suppose that $A$ is not homogeneous. 
As in Hypothesis~\ref{xxhyp4.1}, write $A=\bigoplus_{i\in I} A_i$ where each $A_i$ is homogeneous. Write $M=\bigoplus_{i\in I} M_i$ where each $M_i$ is a left $A_i$-module. 
By the assertions in the homogeneous case, for each $i \in I$, we have 
\begin{enumerate}
\item[($1_i$)]
If $\projdim M_i<\infty$, 
\[
\projdim M_i= d_i-\depth M_i,
\]
where $d_i=\injdim A_i$.
\item[($2_i$)]
If $\injdim M_i<\infty$, then
\[
\injdim M_i= d_i.
\]
\item[($3_i$)]
For every integer $s$ between $\depth M_i$ and 
$\injdim M_i$,
there is a simple left $A_i$-module $T$ such that $\Ext^s_{A_i}(T,M_i)\neq 0$.
\end{enumerate}

(1) Suppose that $\projdim M<\infty$. Then $\projdim M=\max_{i\in I}\{\projdim M_i\}$
and $\depth M=\min_{i\in I}\{\depth M_i\}$. Choose $j\in I$ such that
$\projdim M=\projdim M_j$. Then 
\[
\projdim M=\projdim M_j=d_j-\depth M_j\leq d-\depth M
\]
as desired.

(2) Let $j\in I$ satisfy $\injdim M=\injdim M_j$. Then
\[
\injdim M=\injdim M_j=d_j\leq d
\]
as desired.

(3) Since we are in the case that $A$ is not homogeneous, $M$ is 
indecomposable by hypothesis. In this case $M=M_j$ for some 
$j$. But then $M$ is an $A_j$-module and $A_j$ is homogeneous, so the assertion follows by 
\cite[Theorem 0.1(3)]{WZ2}.
\end{proof}

\section{Weak Hopf algebras}
\label{xxsec5}
Our goal is to apply the main results in Sections~\ref{xxsec3} and 
\ref{xxsec4} to weak Hopf algebras. In this section, we recall 
the definition of a weak Hopf algebra and give some examples. 
Throughout, we use Sweedler notation. If $(C, \Delta, 
\epsilon)$ is a coalgebra, then for $c \in C$, we write 
$\Delta(c) = \sum c_1 \otimes c_2$. When there is no danger of confusion, 
we suppress the summation notation and simply write $\Delta(c) = c_1 \otimes c_2$.

\begin{definition}\cite{BNS, Ha2, Ha3}
\label{xxdef5.1}
A \emph{weak bialgebra} is a $5$-tuple $(H, \mu, u, \Delta, \epsilon)$, 
where $(H , \mu, u)$ is a unital associative $\kk $-algebra and 
$(H, \Delta, \epsilon)$ is a counital coassociative $\kk $-coalgebra,  
satisfying
\begin{enumerate}
\item[(a)] (multiplicativity of comultiplication)
\[
\Delta(ab) = \Delta(a)\Delta(b)
\]
for all $a, b \in H$,
\item[(b)] (weak comultiplicativity of the unit)
\[
\Delta^2(1) = (\Delta(1) \otimes 1)(1 \otimes \Delta(1)) = 
(1 \otimes \Delta(1))(\Delta(1) \otimes 1),
\]
\item[(c)] (weak multiplicativity of the counit)
\[
\epsilon(abc) = \epsilon(ab_1)\epsilon(b_2c) = \epsilon(ab_2)\epsilon(b_1c)
\]
for all $a,b , c \in H$.
\end{enumerate}
We do not assume that $H$ is finite-dimensional over $\kk$.
\end{definition}

Because the coproduct does not necessarily preserve the unit, applying the usual 
sumless Sweedler notation we write $\Delta(1) = 1_1 \otimes 1_2$.
A weak bialgebra is a bialgebra if and only if $\Delta(1) = 1 \otimes 1$, if 
and only if $\epsilon(ab) = \epsilon(a)\epsilon(b)$ for all $a,b \in H$.
For a weak bialgebra $H$, the \emph{source counital map} $\epsilon_s: H \to H$ 
is defined by $\epsilon_s(h) = 1_1 \epsilon(h1_2)$ for all $h\in H$. The 
\emph{source counital subalgebra} $H_s$ is defined to be the image of 
$\epsilon_s$. Similarly, the \emph{target counital map} $\epsilon_t: H \to H$ 
is defined by $\epsilon_t(h) = \epsilon(1_1 h) 1_2$ for all $h\in H$, and the 
\emph{target counital subalgebra} is $H_t = \epsilon_t(H)$. The counital 
subalgebras are finite-dimensional, separable, coideal subalgebras of $H$ 
that commute with each other \cite[Propositions 2.2.2 and 2.3.4]{NV}.

\begin{definition} \cite{BNS, Ha2, Ha3}
\label{xxdef5.2}
Let $H$ be a weak bialgebra.
\begin{enumerate}
\item[(1)]
$H$ is a \emph{weak Hopf algebra} if there exists an algebra 
antihomomorphism $S: H \to H$ called the \emph{antipode} satisfying, for all 
$a \in H$:
\begin{align*}
S(a_1)a_2 &= \epsilon_s(a),\\
a_1S(a_2) &= \epsilon_t(a), \\ 
S(a_1)a_2S(a_3) &= S(a).
\end{align*}
\item[(2)]
A \emph{morphism} between weak Hopf algebras $H_1$ and $H_2$ with 
antipodes $S_1$ and $S_2$ is a map $f: H_1 \to H_2$ which is a 
unital algebra homomorphism and a counital coalgebra homomorphism 
satisfying $f \circ S_1 = S_2 \circ f$.
\end{enumerate}
\end{definition}

There are many examples of finite-dimensional weak Hopf algebras in 
the literature, see for example, \cite{BCJ, BNS, Ha1, Ha2, Ha3, NV, Sz}. We now provide several examples of weak Hopf algebras.

\begin{example}
\label{xxex5.3} 
\begin{enumerate}
\item[(1)]
Every Hopf algebra is a weak Hopf algebra.
\item[(2)]
If $H_1$ and $H_2$ are weak Hopf algebras, then so are 
$H_1\oplus H_2$ and $H_1\otimes H_2$, with their usual algebra and coalgebra 
structures.
\item[(3)]
For every positive integer $n$, the matrix algebra $M_n(\kk)$ 
is a weak Hopf algebra, which is a special case of a groupoid 
algebra. As a consequence of part (2), if $H$ is a weak Hopf 
algebra, so is $M_n(H)$.
\end{enumerate}
\end{example}
In contrast to the examples above, note that if $H_1$ and $H_2$ are Hopf algebras, then $H_1\oplus H_2$ and $M_n(H_1)$ (for $n>1$) are not Hopf algebras. 

The {\it face algebras} defined 
by Hayashi \cite{Ha3} are a special class of weak Hopf algebras. Other 
examples include groupoid algebras and their duals, Temperley-Lieb 
algebras, and quantum transformation groupoids (see \cite{NV}). The focus of 
this paper is on infinite-dimensional weak Hopf algebras (of finite 
GK dimension). In addition to the Hopf algebras of finite positive 
GK dimension, we also provide the following examples and constructions.

\begin{example}
\label{xxex5.4}
Suppose that $(W, \mu_W, u_W, \Delta_W, \epsilon_W,S_W)$ is a weak Hopf 
algebra and that $\sigma$ is a weak Hopf algebra automorphism 
of $W$. Then the group $\mathbb{Z} = \langle a \rangle$ acts on $W$ 
via $\sigma$, so $W$ is a $\kk \mathbb{Z}$-module algebra. Hence, 
we may form the smash product $H = W \# \kk \mathbb{Z}$. 

As a vector space, $H = W \otimes \kk \mathbb{Z}$. As an algebra, 
\[ (w \otimes a^m)(v \otimes a^n) = w \sigma^{m}(v) \otimes a^{m+n}
\]
for $w , v \in W$ and $m,n \in \mathbb{Z}$. The unit of $H$ is given by
$1_{H} = 1_W \otimes a^0$. As a coalgebra, define 
$\Delta(w \otimes a^m) = (w_1 \otimes a^m) \otimes (w_2 \otimes a^m)$ 
and $\epsilon(w \otimes a^m) = \epsilon_W(w)\epsilon_{k \mathbb{Z}}(a^m) 
= \epsilon_W(w)$ for all $w \in W$ and $m \in \mathbb{Z}$. We claim that this makes 
$H$ a weak bialgebra. 

Note that since $\sigma$ 
is a weak Hopf algebra automorphism of $W$, we have that for all 
$w \in W$, $\epsilon(\sigma(w)) = \epsilon(w)$ and 
$\sigma(w)_1 \otimes \sigma(w)_2 = \sigma(w_1) \otimes \sigma(w_2)$.
First, $\Delta$ is multiplicative as
\begin{align*} 
\Delta((w \otimes a^m)(v \otimes a^n)) &= \Delta(w \sigma^m(v) \otimes a^{m+n}) \\
&= ((w \sigma^m(v))_1 \otimes a^{m+n}) \otimes ((w \sigma^m(v))_2 \otimes a^{m+n}) \\
&=(w_1\sigma^m(v)_1 \otimes a^{m+n}) \otimes (w_2 \sigma^m(v)_2 \otimes a^{m+n}) \\
&=(w_1\sigma^m(v_1) \otimes a^{m+n}) \otimes (w_2 \sigma^m(v_2) \otimes a^{m+n}) \\
&=((w_1 \otimes a^m) \otimes (w_2 \otimes a^m) ) ((v_1 \otimes a^{n}) \otimes (v_2 \otimes a^{n})) \\
&=\Delta(w \otimes a^m)\Delta(v \otimes a^n).
\end{align*}

We also have that the unit is weakly comultiplicative, since
\begin{align*} 
\Delta^2(1_{H}) &= (1_1 \otimes a^0) \otimes (1_2 \otimes a^0) \otimes (1_3 \otimes a^0) \\
&= ((1_1 \otimes a^0) \otimes (1_2 \otimes a^0) \otimes (1_W \otimes a^0) ) 
((1_W \otimes a^0) \otimes (1'_1 \otimes a^0) \otimes (1'_2 \otimes a^0)  )
\\
&= (\Delta(1_H) \otimes 1_H)(1_H \otimes \Delta(1_H) ),
\end{align*}
and similarly for the other weak comultiplicativity axiom. 

Finally, to 
see that $\epsilon$ is weakly multiplicative,  we have 
\begin{align*}
\epsilon((u \otimes a^l) (v \otimes a^m) (w \otimes a^n)) &
  = \epsilon(u \sigma^{l}(v) \sigma^{l+m}(w) \otimes a^{l + m+ n}) \\
&= \epsilon_W(u \sigma^{l}(v) \sigma^{l+m}(w))\\
&= \epsilon_W(u \sigma^{l}(v)_1)\epsilon_W(\sigma^{l}(v)_2 \sigma^{l+m}(w)) \\
&= \epsilon_W(u \sigma^{l}(v)_1)\epsilon_W(\sigma^{l}(v_2\sigma^{m}(w))) \\
&= \epsilon_W(u \sigma^{l}(v_1))\epsilon_W(v_2 \sigma^{m}(w)) \\
&= \epsilon((u \otimes a^l) (v_1 \otimes a^m)) \epsilon((v_2 \otimes a^m) (w \otimes a^n)) \\
&= \epsilon((u \otimes a^l) (v \otimes a^m)_1) \epsilon((v \otimes a^m)_2 (w \otimes a^n)),
\end{align*}
and similarly for the other weak multiplicativity axiom.  Thus $H$ is a weak bialgebra, as claimed.

It is now easy to check that we can give $H$ the structure of a weak Hopf algebra by defining the 
antipode as
\begin{gather*}
S(w \otimes a^m) = (1 \otimes S_{k \mathbb{Z}}(a^m))(S_W(w) \otimes a^0) = (1 \otimes a^{-m})(S_W(w) \otimes a^0) \\
= \sigma^{-m}(S_W(w)) \otimes a^{-m}.
\end{gather*}

As an algebra, $H$ is isomorphic to the skew Laurent ring $W[t^{\pm 1}; \sigma]$, where 
$w \in W$ corresponds to $w \otimes a^0$ and $t$ corresponds to $1 \otimes a$. 
Hence, when $W$ is finite-dimensional, $H$ is a weak Hopf algebra of GK dimension 1. 
Under this isomorphism, $\Delta(t) = \Delta(1)(t \otimes t)$, so $t$ is group-like 
in the sense of \cite{NV}. 
\end{example}

\begin{example}
\label{xxex5.5}
Let $W$ be a weak Hopf algebra.  Let $\sigma: W \to W$ be an algebra automorphism 
and assume that (i) there is an algebra homomorphism $\chi: W \to \kk$ such that $\sigma(w) = \sum \chi(w_1) w_2 = \sum w_1 \chi(w_2)$ for all $w \in W$, that is, $\sigma$ is both a left and right winding homomorphism of some character $\chi$; and (ii)  the antipode $S$ of $W$ satisfies $S = \sigma S \sigma$.

Then by \cite[Theorem 4.4]{LSL}, there is a unique weak Hopf algebra structure 
on the Ore extension $H = W[t; \sigma]$ such that $\Delta_H$, $\epsilon_H$, and $S_H$ restrict on $W$ 
to $\Delta_W$, $\epsilon_W$, and $S_W$ respectively and $t$ is primitive in the weak sense, in other words,  $\Delta(t) = \Delta(1)( 1 \otimes t + t \otimes 1)$; $\epsilon(t) = 0$; and $S(t) = -t$.

Note that if $W$ is finite-dimensional, then $H$ has GK dimension $1$.
\end{example}

\begin{remark}
\label{xxex5.6}
In fact, Lomp, Sant'Ana, and Leite dos Santos  give necessary and 
sufficient conditions in \cite[Theorem 4.4]{LSL} for the existence of a weak Hopf algebra 
structure on a more general Ore extension $W[t; \sigma, \delta]$ which extends
the weak Hopf algebra structure on $W$, under the assumption that 
$t$ is weakly skew-primitive. When $W$ is finite-dimensional, these 
Ore extensions give additional examples of weak Hopf algebras of 
GK dimension 1 which generalize the previous example.  
\end{remark}

Applying the  constructions in the previous examples repeatedly, we can obtain many 
different weak Hopf algebras of positive GK dimension.

\section{An analog of (L1) and (R1) for monoidal categories}
\label{sec:tc}

It is well-known that if $H$ is a weak Hopf algebra, then there is a monoidal product endowing $H\Modleft$ with the structure of a monoidal category. In this section, we study analogs of (L1) and (R1) in the more general setting of monoidal categories, and then apply these results to the special case of modules over a weak Hopf algebra in the next section.  The reader can find the basic definitions of monoidal categories and 
related concepts in \cite{EGNO}.
\begin{definition}
\label{def:star}
Let $\mc{C}$ be a $\kk$-linear abelian category. We say that $\mc{C}$ satisfies (C1) if,
for all projective objects $P \in \mc{C}$ and all $i \geq 0$, the functor $\Ext_{\mc{C}}^i(-, P)$ is exact 
on the subcategory of objects of finite length in $\mc{C}$.
\end{definition}

In this definition $\Ext$ means Yoneda Ext, though in our main intended application where $\mc{C} = H\Modleft$ 
for an algebra $H$, the $\Ext$ functors can be computed with projective or injective resolutions, as usual.
We want to show that (C1) follows from quite general hypotheses when $\mc{C}$ is 
a monoidal category. For the rest of this section, suppose that the $\kk$-linear abelian category $\mc{C}$ is also a monoidal category, with bilinear, biexact tensor product denoted $\otensor$. Let $\mathbbm{1}$ denote the unit object of $\mc{C}$.

\begin{definition}
\label{def:internal}
Suppose that for some object $M \in \mc{C}$, the functor 
$- \otensor M$ has a right adjoint. 
We write $\ohom(M, -)$ for the right adjoint. 
By definition, for all $P, N \in \mc{C}$ there is an adjoint isomorphism
\[
\Phi_{P,M,N}: \Hom_{\mc{C}}(P \otensor M, N) \to \Hom_{\mc{C}}(P, \ohom(M, N))
\]
which is natural in $P$ and $N$. The object $\ohom(M,N)$ is called the \emph{internal Hom} from $M$ to $N$.
\end{definition}

Another way of describing the definition above is as follows: $\ohom(M, N) \in \mc{C}$ is 
the object representing the functor $\Hom_{\mc{C}}(- \otensor M, N)$, when that functor is representable.
If $\mc{D}$ is a full subcategory of $\mc{C}$ such that $- \otensor M$ has a right 
adjoint for all $M \in \mc{D}$, then $\ohom(M, N)$ is defined for all $M \in \mc{D}, N \in \mc{C}$.
It is easy to see in this case that $\ohom(-, -)$ is a bifunctor $\mc{D} \times \mc{C} \to \mc{C}$, which 
is contravariant in the first coordinate.

We have the following version of Freyd's adjoint functor theorem.
\begin{lemma}
\label{lem:freyd}
Assume that $\mc{C}$ is cocomplete, in other words that $\mc{C}$ has (small) direct sums, and that $\mc{C}$ has a generator. 
Then the functor $- \otensor M: \mc{C} \to \mc{C}$ has a right adjoint if and only if it commutes with all (small) direct sums in $\mc{C}$, 
for all $M \in \mc{C}$.
\end{lemma}

Recall the definition of a left dual $V^*$ of an object $V \in \mc{C}$ \cite[Definition 2.10.1]{EGNO}.
\begin{lemma}
\label{lem:dualiso2}
Let $\mc{C}$ be as above. Suppose that $V \in \mc{C}$ 
has a left dual $V^*$.
\begin{enumerate}
\item[(1)] The functor $ - \otensor V^*$ is a right adjoint 
of $- \otensor V$. In particular, the internal Hom, $\ohom(V, N)$, exists 
for all $N \in \mc{C}$, and $\ohom(V, N) \cong N \otensor V^*$.

\item[(2)] The functor $\ohom(V, -)$ is exact.

\item[(3)] For any projective object $P \in \mc{C}$, the object $P \otensor V$ is also projective.
\end{enumerate}
\end{lemma}
\begin{proof}

(1) See \cite[Proposition 2.10.8]{EGNO}.

(2) This follows from part (1), since $\ohom(V, -) \cong - \otensor V^*$ and $\otensor$ is biexact.

(3) This is the same proof as in \cite[Proposition 4.2.12]{EGNO}. We have an isomorphism of functors 
$\Hom_{\mc{C}}(P \otensor V, -) \cong \Hom_{\mc{C}}(P, - \otensor V^*)$ by part (1). Since $\otensor$ is biexact and $P$ is projective, the second functor is exact. So the first functor is exact, implying that $P \otensor V$ is projective.
\end{proof}

We now get the following adjunction at the level of Ext.
\begin{lemma}
\label{lem:extadjunction2}
Suppose that $\mc{C}$ has enough projectives. Let $\mc{D}$ be a full subcategory of $\mc{C}$ such that 
every $V \in \mc{D}$ has a left dual $V^*$.

Then for all $M, N \in \mc{C}$ and $V \in \mc{D}$ there is a vector space isomorphism
\[
\Ext^i_{\mc{C}}(M, \ohom(V, N)) \cong \Ext^i_{\mc{C}}(M \otensor V, N),
\]
which is natural in $M$, $V$, and $N$.
\end{lemma}
\begin{proof}
This is similar to the proof in \cite[Proposition 1.3]{BG}. Let $P^{\bullet}$ be a projective 
resolution of $M$. By Lemma~\ref{lem:dualiso2} (3), the complex $P^{\bullet} \otensor V$ consists of projectives, and since $ - \otensor V$ is exact, it forms a projective resolution of $M \otensor V$.
Then $\Hom_{\mc{C}}(P^{\bullet} \otensor V, N)$ has homology groups $\Ext^i_{\mc{C}}(M \otensor V, N)$. On the other hand, using that 
$- \otensor V$ has a right adjoint by Lemma~\ref{lem:dualiso2}(1), 
this complex is isomorphic to $\Hom_{\mc{C}}(P^{\bullet}, \ohom(V, N))$, which has homology groups $\Ext^i_{\mc{C}}(M, \ohom(V, N))$.
The naturality is easy to check.
\end{proof}

The main result of this section is the following version of \cite[Lemma 1.11]{BG} and \cite[Lemma 3.4]{WZ1}.
\begin{proposition}
\label{lem:propexact}
let $\mc{C}$ be as above.  Suppose that $\mc{C}$ has enough projectives. Let $\mc{D}$ be a full abelian subcategory of $\mc{C}$ such that 
every $V \in \mc{D}$ has a left dual $V^* \in \mc{D}$.
\begin{enumerate}
\item[(1)] Suppose $Q \in \mc{C}$ is projective and let $V \in \mc{D}$.
Then $P = Q \otensor V^*$ is projective and for any $i \geq 0$, there is an isomorphism $\Ext^i_{\mc{C}}(V, Q) \cong \Ext^n_{\mc{C}}(\mathbbm{1}, Q \otensor V^*)$, which is natural in $V \in \mc{D}$.

\item[(2)] For any projective object $Q \in \mc{C}$ and any $i \geq 0$, the functor $\Ext^i_{\mc{C}}(-, Q)$ is a contravariant exact functor $\mc{D} \to \kk \Modleft$.
\end{enumerate}
\end{proposition}
\begin{proof}
(1) By Lemma~\ref{lem:dualiso2}(1), there is an isomorphism $Q \otensor V^* \to \ohom(V, Q)$, which it is straightforward to check is natural in $V \in \mc{D}$.  Moreover, by Lemma~\ref{lem:dualiso2}(3), $P = Q \otensor V^*$ is projective.
Now by the adjoint isomorphism in Lemma~\ref{lem:extadjunction2}, we see that 
\[
\Ext^i_{\mc{C}}(\mathbbm{1}, Q \otensor V^*) \cong \Ext^i_{\mc{C}}(\mathbbm{1}, \ohom(V, Q)) \cong \Ext^i_{\mc{C}}(\mathbbm{1} \otensor V, Q) \cong \Ext^i_{\mc{C}}(V, Q),
\]
where all of the displayed isomorphisms are natural in $V \in \mc{D}$.

(2) Let $0 \to V_1 \to V_2 \to V_3 \to 0$ be an exact sequence in $\mc{D}$. By 
assumption, $V_1$, $V_2$, and $V_3$ all have left duals. Note that the functor which assigns to an object its left dual is an exact functor, by the same proof as in \cite[Proposition 4.2.9]{EGNO}. We therefore have an exact sequence $0 \to V_3^* \to V_2^* \to V_1^* \to 0$ of objects in $\mc{D}$. 
Since $\otensor$ is biexact, $0 \to Q \otensor V_3^* \to Q \otensor V_2^* \to Q \otensor V_1^* \to 0$ is an exact sequence as well. 
As the terms of this sequence are all projective by Lemma~\ref{lem:dualiso2}, the sequence is split. Then 
setting $P_i = Q \otensor V_i^*$, we have 
\[
0 \to \Ext^i_{\mc{C}}(\mathbbm{1}, P_3) \to \Ext^i_{\mc{C}}(\mathbbm{1}, P_2) \to \Ext^i_{\mc{C}}(\mathbbm{1}, P_1) \to 0
\]
is exact, since $\Ext^i_{\mc{C}}$ commutes with finite direct sums in the second coordinate. By the isomorphism in part (1), 
\[
0 \to \Ext^i_{\mc{C}}(V_3, Q) \to \Ext^i_{\mc{C}}(V_2, Q) \to \Ext^i_{\mc{C}}(V_1, Q) \to 0
\]
is exact, as required.
\end{proof}

\begin{corollary}
\label{cor:l1}
Suppose that $\mc{C}$ is a $\kk$-linear abelian monoidal category with bilinear, biexact tensor product $\otimes$.  Suppose that $\mc{C}$ has enough projectives. Let $\mc{D}$ be the full subcategory of $\mc{C}$ consisting of finite length objects, and suppose that every $V \in \mc{D}$ has a left dual $V^* \in \mc{D}$. Then $\mc{C}$ satisfies the condition {\rm (C1)} of Definition~\ref{def:star}.
\end{corollary}

\section{Proofs of Theorems~\ref{xxthm0.2}--\ref{xxthm0.7}}
\label{xxsec7}

Let $H$ be a weak Hopf algebra. 
The monoidal product on $H \Modleft$ is defined as follows.
For $M,N \in H \Modleft$, $M \otimes_k N$ is a (non-unital) left $H$-module 
via the action $h \cdot (m \otimes n) = \sum h_1m \otimes h_2n$.  Then one defines 
\[
M \otensor N := \Delta(1) (M \otimes N) = \big\{ 1_1m \otimes 1_2 n \mid m \in M, n \in N \big\},
\]
which is a unital submodule of $M \otimes N$ and thus in $H \Modleft$.  
This makes $H \Modleft$ into a monoidal category \cite[Section 5.1]{NV}.

There is also 
another way to describe this product which is often convenient.  Recall 
the counital subalgebra $H_t = \epsilon_t(H)$, where $\epsilon_t(h) = \epsilon(1_1 h) 1_2$.
Let us define also $\overline{\epsilon_s}(h) = 1_1 \epsilon(1_2 h)$.
If $M \in H \Modleft$, then of course $M$ is also a left $H_t$-module by restriction.
The module $M$ also has a right $H_t$-structure, where for $m \in M$, $h \in H_t$ we define $m \cdot h := \overline{\epsilon_s}(h)m$.  Under these two actions, $M$ becomes an $(H_t, H_t)$-bimodule.  Then we can also identify $M \otensor N$ with $M \otimes_{H_t} N$, with 
the same formula for the left $H$-action, that is $h \cdot (m \otimes n) =  h_1 m \otimes h_2 n$ \cite[Theorem 2.4]{BCJ}.

We are now ready to prove the main results given in the Introduction.  

\begin{proof}[Proof of Theorem~\ref{xxthm0.2}]
(1) Let $A$ be an algebra which satisfies Hypotheses~\ref{xxhyp0.1} and \ref{xxhyp0.2}.
Since $A$ is a finitely generated module over its affine center, by \cite[Proposition 3.1]{BG}, every finite length $H$-module is finite-dimensional over $\kk$. Since $A$ satisfies Hypothesis~\ref{xxhyp0.2}, by Corollary~\ref{cor:l1}, both $A\Modleft$ and $A^{op}\Modleft$ satisfy (C1) of Definition~\ref{def:star} and so $A$ satisfies (L1) and (R1).
 The assertions follow from Corollary~\ref{xxcor3.7}.

(2) Let $H$ be a weak Hopf algebra. Hypothesis~\ref{xxhyp0.2} for $H$ follows from well-known results, though we will briefly review the details. The existing references sometimes assume a weak Hopf algebra is finite-dimensional, but since we only claim that finite-dimensional objects have left duals, there is no significant change to the proofs.

The category $H \Modleft$ is obviously abelian, and it is monoidal with the operation $\otensor$ defined at the beginning of this section. The unit object is $\mathbbm{1} = H_t$, which is a left $H$-module with action $h \cdot z = \epsilon_t(hz)$ for $h \in H, z \in H_t$ \cite[Lemma 5.1.1]{NV}.  By the description of $\otensor$ as $\otimes_{H_t}$, where each $M \in H \Modleft$ has a canonical $(H_t, H_t)$-bimodule structure, since $H_t$ is semisimple, it is clear that $\otensor$ is bilinear on morphisms and biexact.  If $M \in H \Modleft$ is finite-dimensional, it has a left dual $M^* = \Hom_\kk(M, \kk)$ with $H$-module 
action $[h \cdot \phi](m) = \phi(S(h)m)$, using the antipode $S$ of $H$ \cite[Lemma 5.1.2]{NV}.  Note that we have not assumed that $S$ is bijective, and the proof 
of \cite[Lemma 5.1.2]{NV} uses $S^{-1}$.  But in fact $S^{-1}$ is only applied in the proof to elements of $H_s$, and for any weak Hopf algebra the antipode $S$ gives a bijection 
from $H_t$ to $H_s$ \cite[Theorem 2.10]{BNS}, so that $S^{-1} \vert_{H_s}$ makes sense.

This shows that the category $H \Modleft$ has the properties needed in Hypothesis~\ref{xxhyp0.2}.  Since $(H^{op,cop},S)$ is also a weak Hopf algebra (see \cite[Remark 2.4.1]{NV}), the category $\Modright H$ also has these properties, so that Hypothesis~\ref{xxhyp0.2} holds.
\end{proof}

\begin{proof}[Proof of Theorem~\ref{xxthm0.3}]
By Theorem~\ref{xxthm0.2}, $A$ satisfies Hypothesis
\ref{xxhyp4.1}. The assertion follows from Lemma 
\ref{xxlem4.2}.
\end{proof}

\begin{proof}[Proof of Theorem~\ref{xxthm0.4}]
By Theorem~\ref{xxthm0.2}, $A$ satisfies Hypothesis
\ref{xxhyp4.1}. The assertion follows from Lemma 
\ref{xxlem4.3}.
\end{proof}

\begin{proof}[Proof of Theorem~\ref{xxthm0.5}]
By Theorem~\ref{xxthm0.2}, $A$ satisfies Hypothesis
\ref{xxhyp4.1}. The assertion follows from Lemma 
\ref{xxlem4.4}.
\end{proof}

\begin{proof}[Proof of Theorem~\ref{xxthm0.6}]
(1) The assertion follows from Lemma~\ref{xxlem3.5}(1)
and Corollary~\ref{xxcor3.8}.

(2) This follows from Lemma~\ref{xxlem3.5}(2).

(3) By Lemma~\ref{xxlem3.5} and Corollary~\ref{xxcor3.8},
$A$ has a residue complex $K$ that has nonzero cohomology only at $H^{-d}(K)=\Omega$, where $d=\GKdim A$.
By Corollary~\ref{xxcor3.8}, $\Omega$ is an invertible $A$-bimodule.
It follows from the definition of residue complex that 
$\Omega$ has a minimal pure injective resolution on both sides. 
By tensoring with $\Omega^{-1}$ we obtain that $A$
has a minimal pure injective resolution on both sides. 
\end{proof}

\begin{proof}[Proof of Theorem~\ref{xxthm0.7}]
By Theorem~\ref{xxthm0.2}, $A$ satisfies Hypothesis
\ref{xxhyp4.1}. The assertion follows from Proposition
\ref{xxpro4.6}.
\end{proof}

In this paper we have focused on the class of weak Hopf algebras, which seems to be an especially fertile ground for generalizing the homological theory of infinite dimensional Hopf algebras. The conditions in Hypothesis~\ref{xxhyp0.2} are quite weak, however, and so we expect other generalizations of Hopf algebras to satisfy the analog of Theorem~\ref{xxthm0.2}(2).  We have not attempted to catalog all of the structures for which this theory applies, but mention one such further example here.

A \emph{quasi-bialgebra} is a generalization of a bialgebra for which the coproduct is not coassociative, but satisfies a weaker form of coassociativity up to twisting by a unit.
There is a natural notion of antipode $S$ for a quasi-bialgebra.  Such algebras arise naturally in the theory of tensor categories; we refer to \cite[Sections 5.13-5.15]{EGNO} for the definition and some basic properties.  Traditionally, the term quasi-Hopf algebra is reserved for quasi-bialgebras with invertible antipode.

Given a quasi-bialgebra $H$, the category $H \Modleft$ is monoidal, where for $M, N \in H \Modleft$ we have $M \otensor N = M \otimes_\kk N$ with action  $h \cdot (m \otimes n) = h_1 m \otimes h_2 n$, similarly as for Hopf algebras.  In particular, it is clear that $\otensor$ is bilinear on morphisms and biexact.  The existence of an antipode $S$ implies that every finite dimensional $M \in H \Modleft$ has a left dual $M^*$, with $H$-action given by the same formula as in the Hopf case \cite[p. 113]{EGNO}.  Analogous results hold for the category of right modules, since $H^{op, cop}$ is also a quasi-Hopf algebra with antipode $S$. Thus we have the following result.
\begin{theorem}
\label{xxthm-qh}
Let $H$ be a quasi-bialgebra with antipode.  Then $H$ satisfies Hypothesis~\ref{xxhyp0.2}, and hence if $H$ is also finite over an affine center than $H$ satisfies all of the conclusions of Theorem~\ref{xxthm0.2}(1).
\end{theorem}

\section{Further questions}
\label{xxsec8}
We conclude by posing some further questions for weak Hopf algebras.
For open questions in the Hopf case, we refer the reader to \cite{Br1, BZ} and the
survey papers \cite{Br2, Go}.

The main result of this paper (Theorem~\ref{xxthm0.2}) is a proof of the Brown--Goodearl Conjecture for weak Hopf algebras 
that satisfy Hypothesis~\ref{xxhyp0.1}. 
It is natural to ask if the Brown--Goodearl Conjecture holds for weak Hopf algebras satisfying weaker hypotheses. In particular, we ask
\begin{question}[The Brown--Goodearl Question for weak Hopf algebras]
\label{xxque8.3}
Let $H$ be a noetherian weak Hopf algebra. Does $H$ have finite injective dimension? What if we assume further that $H$ is affine and PI?
\end{question}

Again motivated by Theorem~\ref{xxthm0.2}, we ask
\begin{question}[Decomposition Question]
\label{xxque8.4}
Suppose $H$ is a noetherian weak Hopf algebra with finite
Gelfand--Kirillov dimension $d$. Is then $H$ a direct
sum of homogeneous weak Hopf subalgebras $H_i$ of 
Gelfand--Kirillov dimension $i$ for integers $0 \leq i\leq d$?
\end{question}

In \cite{WZ1}, Wu and the third-named author proved that noetherian affine PI Hopf algebras have artinian quotient rings. As a corollary of \cite[Theorem A]{Sk}, Skryabin deduced that these Hopf algebras have a bijective antipode.
In light of Theorem~\ref{xxthm0.3}, it is therefore natural to ask
\begin{question}[Skryabin's Question for weak Hopf algebras]
\label{xxque8.9}
Does every noetherian weak Hopf algebra $H$ have a bijective antipode? What if we assume further that $H$ is affine and PI?
\end{question}

\bibliographystyle{alpha}
\bibliography{biblio}{}

\begin{thebibliography}{{Hay}99a}

\bibitem[AM69]{AM}
M.~F. Atiyah and I.~G. Macdonald.
\newblock {\em Introduction to commutative algebra}.
\newblock Addison-Wesley Publishing Co., Reading, Mass.-London-Don Mills, Ont.,
  1969.

\bibitem[And04]{An}
Nicol\'{a}s Andruskiewitsch.
\newblock Some remarks on {N}ichols algebras.
\newblock In {\em Hopf algebras}, volume 237 of {\em Lecture Notes in Pure and
  Appl. Math.}, pages 35--45. Dekker, New York, 2004.

\bibitem[ASZ98]{ASZ1}
K.~Ajitabh, S.~P. Smith, and J.~J. Zhang.
\newblock Auslander-{G}orenstein rings.
\newblock {\em Comm. Algebra}, 26(7):2159--2180, 1998.

\bibitem[BCJ11]{BCJ}
G.~B\"{o}hm, S.~Caenepeel, and K.~Janssen.
\newblock Weak bialgebras and monoidal categories.
\newblock {\em Comm. Algebra}, 39(12):4584--4607, 2011.

\bibitem[BG97]{BG}
K.~A. Brown and K.~R. Goodearl.
\newblock Homological aspects of {N}oetherian {PI} {H}opf algebras of
  irreducible modules and maximal dimension.
\newblock {\em J. Algebra}, 198(1):240--265, 1997.

\bibitem[BNS99]{BNS}
Gabriella B\"ohm, Florian Nill, and Korn\'el Szlach\'anyi.
\newblock Weak {H}opf algebras. {I}. {I}ntegral theory and {$C^*$}-structure.
\newblock {\em J. Algebra}, 221(2):385--438, 1999.

\bibitem[Bro98]{Br1}
Kenneth~A. Brown.
\newblock Representation theory of {N}oetherian {H}opf algebras satisfying a
  polynomial identity.
\newblock In {\em Trends in the representation theory of finite-dimensional
  algebras ({S}eattle, {WA}, 1997)}, volume 229 of {\em Contemp. Math.}, pages
  49--79. Amer. Math. Soc., Providence, RI, 1998.

\bibitem[Bro07]{Br2}
Kenneth~A. Brown.
\newblock Noetherian {H}opf algebras.
\newblock {\em Turkish J. Math.}, 31(suppl.):7--23, 2007.

\bibitem[BZ08]{BZ}
K.~A. Brown and J.~J. Zhang.
\newblock Dualising complexes and twisted {H}ochschild (co)homology for
  {N}oetherian {H}opf algebras.
\newblock {\em J. Algebra}, 320(5):1814--1850, 2008.

\bibitem[EGNO15]{EGNO}
Pavel Etingof, Shlomo Gelaki, Dmitri Nikshych, and Victor Ostrik.
\newblock {\em Tensor categories}, volume 205 of {\em Mathematical Surveys and
  Monographs}.
\newblock American Mathematical Society, Providence, RI, 2015.

\bibitem[ENO05]{ENO}
Pavel Etingof, Dmitri Nikshych, and Viktor Ostrik.
\newblock On fusion categories.
\newblock {\em Ann. of Math. (2)}, 162(2):581--642, 2005.

\bibitem[Goo13]{Go}
K.~R. Goodearl.
\newblock Noetherian {H}opf algebras.
\newblock {\em Glasg. Math. J.}, 55(A):75--87, 2013.

\bibitem[GW04]{GW}
K.~R. Goodearl and R.~B. Warfield, Jr.
\newblock {\em An introduction to noncommutative {N}oetherian rings}, volume~61
  of {\em London Mathematical Society Student Texts}.
\newblock Cambridge University Press, Cambridge, second edition, 2004.

\bibitem[Hay93]{Ha3}
Takahiro Hayashi.
\newblock Quantum group symmetry of partition functions of {IRF} models and its
  application to {J}ones' index theory.
\newblock {\em Comm. Math. Phys.}, 157(2):331--345, 1993.

\bibitem[{Hay}99a]{Ha1}
T.~{Hayashi}.
\newblock {A canonical Tannaka duality for finite seimisimple tensor
  categories}.
\newblock {\em ArXiv Mathematics e-prints}, April 1999.

\bibitem[Hay99b]{Ha2}
Takahiro Hayashi.
\newblock Face algebras and unitarity of {${\rm SU}(N)_L$}-{TQFT}.
\newblock {\em Comm. Math. Phys.}, 203(1):211--247, 1999.

\bibitem[IK10]{IK}
Miodrag~Cristian Iovanov and Lars Kadison.
\newblock When weak {H}opf algebras are {F}robenius.
\newblock {\em Proc. Amer. Math. Soc.}, 138(3):837--845, 2010.

\bibitem[KL00]{KL}
G\"{u}nter~R. Krause and Thomas~H. Lenagan.
\newblock {\em Growth of algebras and {G}elfand-{K}irillov dimension},
  volume~22 of {\em Graduate Studies in Mathematics}.
\newblock American Mathematical Society, Providence, RI, revised edition, 2000.

\bibitem[Lam01]{Lam}
T.~Y. Lam.
\newblock {\em A first course in noncommutative rings}, volume 131 of {\em
  Graduate Texts in Mathematics}.
\newblock Springer-Verlag, New York, second edition, 2001.

\bibitem[Lev92]{Le}
Thierry Levasseur.
\newblock Some properties of noncommutative regular graded rings.
\newblock {\em Glasgow Math. J.}, 34(3):277--300, 1992.

\bibitem[LSLdS19]{LSL}
Christian Lomp, Alveri Sant'Ana, and Ricardo Leite~dos Santos.
\newblock Panov's theorem for weak {H}opf algebras.
\newblock In {\em Rings, {M}odules and {C}odes}, volume 727 of {\em Contemp.
  Math.}, pages 277--292. Amer. Math. Soc., Providence, RI, 2019.

\bibitem[{Nil}98]{Nill}
Florian {Nill}.
\newblock {Axioms for Weak Bialgebras}.
\newblock {\em arXiv Mathematics e-prints}, page math/9805104, May 1998.

\bibitem[NV02]{NV}
Dmitri Nikshych and Leonid Vainerman.
\newblock Finite quantum groupoids and their applications.
\newblock In {\em New directions in {H}opf algebras}, volume~43 of {\em Math.
  Sci. Res. Inst. Publ.}, pages 211--262. Cambridge Univ. Press, Cambridge,
  2002.

\bibitem[Skr06]{Sk}
Serge Skryabin.
\newblock New results on the bijectivity of antipode of a {H}opf algebra.
\newblock {\em J. Algebra}, 306(2):622--633, 2006.

\bibitem[SZ94]{SZ1}
J.~T. Stafford and J.~J. Zhang.
\newblock Homological properties of (graded) {N}oetherian {${\rm PI}$} rings.
\newblock {\em J. Algebra}, 168(3):988--1026, 1994.

\bibitem[Szl01]{Sz}
Korn\'{e}l Szlach\'{a}nyi.
\newblock Finite quantum groupoids and inclusions of finite type.
\newblock In {\em Mathematical physics in mathematics and physics ({S}iena,
  2000)}, volume~30 of {\em Fields Inst. Commun.}, pages 393--407. Amer. Math.
  Soc., Providence, RI, 2001.

\bibitem[vdB97]{VdB}
Michel van~den Bergh.
\newblock Existence theorems for dualizing complexes over non-commutative
  graded and filtered rings.
\newblock {\em J. Algebra}, 195(2):662--679, 1997.

\bibitem[Vec03]{Vec}
Peter Vecserny\'{e}s.
\newblock Larson-{S}weedler theorem and the role of grouplike elements in weak
  {H}opf algebras.
\newblock {\em J. Algebra}, 270(2):471--520, 2003.

\bibitem[WZ01]{WZ2}
Q.-S. Wu and J.~J. Zhang.
\newblock Homological identities for noncommutative rings.
\newblock {\em J. Algebra}, 242(2):516--535, 2001.

\bibitem[WZ03]{WZ1}
Q.-S. Wu and J.~J. Zhang.
\newblock Noetherian {PI} {H}opf algebras are {G}orenstein.
\newblock {\em Trans. Amer. Math. Soc.}, 355(3):1043--1066, 2003.

\bibitem[Yek92]{Ye1}
Amnon Yekutieli.
\newblock Dualizing complexes over noncommutative graded algebras.
\newblock {\em J. Algebra}, 153(1):41--84, 1992.

\bibitem[Yek99]{Ye2}
Amnon Yekutieli.
\newblock Dualizing complexes, {M}orita equivalence and the derived {P}icard
  group of a ring.
\newblock {\em J. London Math. Soc. (2)}, 60(3):723--746, 1999.

\bibitem[YZ99]{YZ4}
Amnon Yekutieli and James~J. Zhang.
\newblock Rings with {A}uslander dualizing complexes.
\newblock {\em J. Algebra}, 213(1):1--51, 1999.

\bibitem[YZ03]{YZ3}
Amnon Yekutieli and James~J. Zhang.
\newblock Residue complexes over noncommutative rings.
\newblock {\em J. Algebra}, 259(2):451--493, 2003.

\end{thebibliography}

\end{document}